\documentclass[10pt,letterpaper]{amsart}
\usepackage{amsmath,amsthm,amscd,amsfonts,amssymb,epic,eepic,bbm,longtable, enumitem, stmaryrd,tikz-cd}
\usepackage[pagebackref,colorlinks=true,linkcolor=blue,citecolor=blue]{hyperref}
\makeatletter
\let\@wraptoccontribs\wraptoccontribs
\makeatother
\allowdisplaybreaks

\newtheorem{thm}{Theorem}[section]
\newtheorem{cor}[thm]{Corollary}

\newtheorem{lem}[thm]{Lemma}

\newtheorem{prop}[thm]{Proposition}
\theoremstyle{remark}

\newcounter{remarkscounter}

\numberwithin{equation}{section}
\newcommand{\A}{\mathbb{A}}
\newcommand{\GL}{\mathrm{GL}}

\newcommand{\SL}{\mathrm{SL}}
\newcommand{\ZZ}{\mathbb{Z}}

\newcommand{\QQ}{\mathbb{Q}}

\newcommand{\lto}{\longrightarrow}

\newcommand{\OO}{\mathcal{O}}
\newcommand{\CC}{\mathbb{C}}
\newcommand{\RR}{\mathbb{R}}
\newcommand{\GG}{\mathbb{G}}

\newcommand{\quash}[1]{}

\theoremstyle{definition}

\renewcommand{\bar}{\overline}

\numberwithin{equation}{section}

\newcommand{\one}{\mathbbm{1}}

\newenvironment{psmatrix}
  {\left(\begin{smallmatrix}}
  {\end{smallmatrix}\right)}

\begin{document}

\title{Summation formulae for quadrics}

\author{Jayce R. Getz}
\address{Department of Mathematics\\
Duke University\\
Durham, NC 27708}
\email{jgetz@math.duke.edu}


\thanks{The author is partially supported by NSF grant DMS 2400550.
Any opinions, findings, and conclusions or recommendations expressed in this material are those of the author and do not necessarily reflect the views of the National Science Foundation.  The author also thanks D.~Kazhdan for travel support under his ERC grant AdG 669655.
}

\maketitle
\begin{abstract}
We prove a Poisson summation formula for the zero locus of a quadratic form
in an even number of variables with no assumption on the support of the functions involved.  The key novelty in the formula is that all ``boundary terms'' are given either by constants or sums over smaller quadrics related to the original quadric.
We also discuss the link with the classical problem of estimating the number of solutions of a quadratic form in an even number of variables.   To prove the summation formula we compute (the Arthur truncated) theta lift of the trivial representation of $\SL_2(\A_F)$.  As previously observed by Ginzburg, Rallis, and Soudry, this is an analogue for orthogonal groups on vector spaces of even dimension of the global Schr\"odinger representation of 
the metaplectic group. 
\end{abstract}

\tableofcontents

\section{Introduction}
\label{sec:intro}
Let $V_0=\GG_a^{2d}$ be an even dimensional affine space over a number field $F$ and let $Q_{0}$ be a nondegenerate anisotropic quadratic form on $V_{0}.$    We allow the degenerate special case where $V_{0}=\{0\},$ equipped with the trivial quadratic form. 
For $i \geq 0$ let
\begin{align} \label{Vi}
V_i:=V_{0} \oplus \GG_a^{2i}.
\end{align}
We equip $V_i$ with the nondegenerate quadratic form $Q_i$ defined in \S \ref{ssec:groups}.  The quadratic form $Q_i$ is in the Witt class of $Q_0.$

For $F$-algebras $R$ let
\begin{align} \label{Xi}
X_i(R):=\{ u \in V_i(R): Q_i(u)=0\}
\end{align}
and let $X_i^{\circ}:=X_i-\{0\}$.  
Fix $\ell \in \ZZ_{> 0}$.  If $V_0=\{0\}$ we always assume that $\ell>1.$
In this paper we prove a summation formula for $X_{\ell}$ analogous to the Poisson summation formula.  It involves the whole family of spaces $X_{i}$ for $\ell \geq i \geq 0$.  
 
\subsection{A summation formula} 
Let $\A_F$ be the adeles of $F.$
  Fix an additive character $\psi:F \backslash \A_F \to \CC^\times.$  We can then define the Weil representation
\begin{align} \label{rhoi}
\rho_i:=\rho_{Q_i,\psi}:\SL_2(\A_F)  \times \mathcal{S}(V_i(\A_F)) \lto \mathcal{S}(V_i(\A_F))
 \end{align}
 as usual. We have suppressed the orthogonal group in the Weil representation because it plays no role at the moment (but see the discussion below Theorem \ref{thm:minimal}).
We have an action 
\begin{align} \label{Lvee} \begin{split}
L^\vee:\SL_2(\A_F) \times \mathcal{S}(\A_F^2) &\lto \mathcal{S}(\A_F^2)\\
(g,f) &\longmapsto \left(v \mapsto f(g^tv)\right)\end{split}
\end{align}
and thus an action 
\begin{align} \label{ri}
r_i:=\rho_i \otimes L^{\vee}:\SL_2(\A_F) \times \mathcal{S}( V_{i}(\A_F) \oplus \A_F^2) \lto \mathcal{S}( V_{i}(\A_F) \oplus \A_F^2).
\end{align}

For $ i >0$ we define the operator
\begin{align} \label{I} \begin{split}
I:\mathcal{S}(V_i(\A_F) \oplus \A_F^2) &\lto C^\infty(X_i^\circ(\A_F))\\
f &\longmapsto \left(\xi \mapsto \int_{N(\A_F) \backslash \SL_2(\A_F)} r_i(g )f \left(\xi,0,1\right)d\dot{g}\right). \end{split}
\end{align} 
Here $N \leq \SL_2$ is the unipotent radical of the Borel subgroup of upper triangular matrices.
The integral $I(f)$ is absolutely convergent and defines a function in $C^\infty(X^{\circ}_i(\A_F))$ by lemmas \ref{lem:conv} and \ref{lem:unr:conv} below.

For $i > i' \geq 0$ we have operators
\begin{align*}
c_i:\mathcal{S}(V_i(\A_F) \oplus \A_F^2) &\lto \CC\\
d_{i,i'}:\mathcal{S}(V_{i}(\A_F) \oplus \A_F^2)& \lto \mathcal{S}(V_{i'}(\A_F) \oplus \A_F^2)
\end{align*}
defined as in \eqref{ci:def}, \eqref{di0}, and \eqref{di}.  Briefly, $c_i(f)$ is the regularized value of $I(f)$ at $0 \in X_i(F).$  On the other hand $d_{i,i-1}:=d_i$ is given by a partial Fourier transform and then restriction to the complement of a hyperbolic plane.  One then sets $d_{i,i'}=d_{i,i-1} \circ \dots \circ d_{i'+1,i'}.$
By convention, $d_{i,i}$ is the identity.  
We let
\begin{align*}
\mathcal{F}_{\wedge}:\mathcal{S}(\A_F^2) \lto \mathcal{S}(\A_F^2)
\end{align*}
be the usual $\SL_2(\A_F)$-equivariant Fourier transform (see \eqref{Fwedge}).  We extend it to $\mathcal{S}(V_{i}(\A_F) \oplus \A_F^2)$ by setting $\mathcal{F}_{X_i}:=  1_{\mathcal{S}(V_i(\A_F))} \otimes \mathcal{F}_{\wedge}$. 

Let $[\SL_2]:=\SL_2(F) \backslash \SL_2(\A_F).$
Our main theorem is the following summation formula:
\begin{thm} \label{thm:main} For $f \in \mathcal{S}(V_\ell(\A_F) \oplus \A_F^2)$ one has that
\begin{align*}
&\sum_{i=1}^\ell\left(c_i(d_{\ell,i}(f))+\sum_{\xi \in X^{\circ}_i(F)}I(d_{\ell,i}(f))(\xi)\right)+ \kappa d_{\ell,0}(f)(0_{V_0},0,0)\\&
=\sum_{i=1}^\ell \left(c_i(d_{\ell,i}(\mathcal{F}_{X_{\ell}} (f)))+\sum_{\xi \in X^{\circ}_i(F)}I(d_{\ell,i}(\mathcal{F}_{X_{\ell}}(f)))(\xi)\right)+\kappa d_{\ell,0}(\mathcal{F}_{X_{\ell}}(f))(0_{V_0},0,0).
\end{align*}
Here 
\begin{align*} 
    \kappa:=\begin{cases}\mathrm{meas}([\SL_2]) &\textrm{ if }\dim V_0=0 \\
    0 &\textrm{ otherwise.}\end{cases}
\end{align*}
\end{thm}
\noindent   The sums over $\xi$ are in general infinite, but they are easily seen to be absolutely convergent by lemmas \ref{lem:conv} and \ref{lem:unr:conv} below.   In general, the $i$th summand on the left and right are not equal. 
 
 The key novelty in Theorem \ref{thm:main} is that it is valid under no restrictions on the test functions involved and all expressions are explicit and evidently geometric in nature.  Poisson summation formulae for $X_\ell$ and other special classes of varieties exist in the literature; we comment on this in the next paragraph.  However Theorem \ref{thm:main} is the first instance of a Poisson summation formula for a singular variety in which every term has a manifestly geometric interpretation.  We should point out though that two years after the current paper was submitted, the paper \cite{GurK:Auto} was written.  In it, the authors use the structure of the minimal representation to prove a Poisson summation formula very similar to Theorem \ref{thm:main} in the function field case.  They also obtain analogous results for minimal representations on the exceptional groups of type $E_n.$  

If $V_\ell$ is a split quadratic space and we place additional assumptions on $f$ then the theorem is a consequence of more general work of Braverman and Kazhdan \cite{BK:normalized}.  Their work was generalized to arbitrary test functions $f$ in \cite{Choie:Getz}, but several terms were given inexplicitly in terms of residues of Eisenstein series.
Related formulae are also established for arbitrary quadratic spaces in \cite{GetzQuad} and in a special case related to Rankin-Selberg products on $\GL_2$ in \cite{Getz:RSMonoid}.  
The relationship between Eisenstein series and the $V_0=\{0\}$ case of Theorem \ref{thm:main} will be made precise in \cite{Hsu:Companion}.  

The interest in general test functions stems from the fact that restrictions on test functions limit the information one can extract from the formula.  In more detail, it is expected, and can be verified in some cases, that the terms not corresponding to the open orbit under a suitable group action
correspond to poles of appropriate $L$-functions.  Hence choosing test functions that eliminate these contributions hides information about the poles of $L$-functions.

   Theorem \ref{thm:main} is very closely related to the circle method for quadratic forms and we feel that it will be useful for questions in analytic number theory.    To make the relationship transparent we discuss the special case where $F=\QQ$ in \S \ref{sec:circle:method} below.  Interestingly, in analytic number theory it is often only the most degenerate terms that are studied.    

We expect that the image of the map $I:\mathcal{S}(V_i(\A_F) \oplus \A_F^2) \to C^\infty(X^{\circ}(\A_F))$ is the Schwartz space of $X(\A_F)$ in either the sense of \cite{BK:normalized,GurK} or the sense of \cite{Getz:Hsu:Leslie} and the map 
$$
\mathcal{F}_{X_\ell}:\mathcal{S}(V_\ell(\A_F) \oplus \A_F^2) \lto \mathcal{S}(V_\ell(\A_F) \oplus \A_F^2)
$$
descends to the Fourier transform on the Schwartz space of $X_i(\A_F)$ defined in these references. 
 This may be proved locally.  When $V_0=\{0\}$ the non-Archimedean part of the theory will be contained in \cite{Hsu:Companion}.
 Moreover, in loc.~cit.~the local analogues of the operators $d_i$ and $c_i$ are examined; it turns out that they may be understood in terms of germ expansions at the origin.  In particular, at least in the non-Archimedean case, the functions $I(\mathcal{F}_{X_i}(f))$, $I(d_{i}(f))$ can be recovered from $I(f)$ and a similar statement is true for the operators $c_i.$ 

In this paper we do not give an alternate formulation of Theorem \ref{thm:main} in terms of $I(f) \in C^\infty(X^{\circ}_\ell(\A_F))$ as opposed to $f \in \mathcal{S}(V_\ell(\A_F) \oplus \A_F^2).$ We do prove in \S \ref{sec:invariance} that, at least after grouping the $c_i$ appropriately, all expressions in Theorem \ref{thm:main} can be written in terms of the space of coinvariants 
\begin{align} \label{SX}
\mathcal{S}(X_{\ell}(\A_F)):=\mathcal{S}(V_\ell(\A_F) \oplus \A_F^2)_{r_\ell(\SL_2(\A_F))}.
\end{align}
We feel that defining the Schwartz space in this manner and focusing on it instead of a space of functions on $X_\ell^{\circ}(\A_F)$ is reasonable, as we now explain.
Let  $\mathcal{S}_{\mathrm{BK}}(X_\ell(\A_F))$ be the Schwartz space of $X_\ell(\A_F),$ defined as in \cite{Choie:Getz} (therein $\mathcal{S}_{\mathrm{BK}}(X_\ell(\A_F))$ is denoted simply by $\mathcal{S}(X_{\ell}(\A_F))$).
We expect that the map
$$
\mathcal{S}(X_{\ell}(\A_F)) =\mathcal{S}(V_\ell(\A_F) \oplus \A_F^2)_{r_\ell(\SL_2(\A_F))} \lto C^\infty(X^{\circ}_{\ell}(\A_F))
$$
induced by $I$ has image equal to $\mathcal{S}_{\mathrm{BK}}(X_{\ell}(\A_F)).$   Let us assume this.
Then the map $\mathcal{S}(X_{\ell}(\A_F)) \to \mathcal{S}_{\mathrm{BK}}(X_{\ell}(\A_F))$ is formally analogous to a morphism from a stack to a coarse moduli space that resolves the singularities of the coarse moduli space.  The analogy is even tighter if one bears in mind Bernstein's definition of the Schwartz space of a smooth algebraic stack \cite{Sak:Stack}.  For many purposes, the stack  is more convenient than the coarse moduli space itself. 
For similar reasons, we prefer to work with $\mathcal{S}(X_{\ell}(\A_F))$ as opposed to  $\mathcal{S}_{\mathrm{BK}}(X_{\ell}(\A_F)).$  Despite this, to alleviate confusion we will not use the notation $\mathcal{S}(X_{\ell}(\A_F))$ until \S \ref{sec:invariance}.  

\subsection{Canonicity} \label{ssec:canon}
By a \textbf{framed flag of quadratic spaces extending $V_0$} we mean a collection of vector spaces
\begin{align*}
V_0<V_1'<\dots<V_{\ell}'
\end{align*}
equipped with a nondegenerate quadratic form $Q'_\ell$ on $V_{\ell}'$ such that $Q_{\ell}'|_{V_0}=Q_0$
together with $v_{i}',w'_{i} \in V_{i+1}'(F) \cap V_i'^{\perp}(F) $ satisfying $Q'_{\ell}(v_i')=Q_{\ell}'(w_{i}')=0$ and $Q'_{\ell}(v_i'+w_{i}')=1.$  
Here $V'^{\perp}_i \subset V_{\ell}'$ is the space orthogonal to $V_i$ with respect to the pairing attached to $Q_{\ell}'.$
The \textbf{standard framed flag of quadratic spaces extending $V_0$} is 
$$
V_0<V_1<\dots<V_\ell
$$
equipped with $e_{\dim V_0+i},e_{\dim V_0+i+1}$ for all $i.$  
  The definition of $I$ in \eqref{I} and $d_i$ and $c_i$ in \eqref{ci:def}, \eqref{di0} and \eqref{di} below depends on the standard framed flag of quadratic spaces.  We could also define maps $I,$ $d_i,$ $c_i$ with respect to another framed flag of quadratic spaces extending $V_0$  and similarly define 
  $$
  X_i'(R):=\{v \in V_i(R):Q_i(v)=0\}.
  $$
  We refer to \S \ref{sec:canon} for details.

In \S \ref{sec:canon} we construct the universal framed flag of quadratic spaces extending $V_0:$ 
$$
V_0 <V_1^{\mathrm{u}}<\dots <V_\ell^{\mathrm{u}}
$$
Let $X_i^{\mathrm{u}} \subset V_i^{\mathrm{u}}$ be the zero locus of the quadratic form $Q_i^{\mathrm{u}}.$ In loc.~cit.~we moreover define operators 
\begin{align*}
c_i^{\mathrm{u}}:\mathcal{S}(V_i^{\mathrm{u}}(\A_F) \oplus \A_F^2) &\lto \CC,  & 1 \leq i \leq \ell-1\\
d_{\ell,i}^{\mathrm{u}}: \mathcal{S}(V_{\ell}'(\A_F) \oplus \A_F^2)& \lto \mathcal{S}(V_{i}^{\mathrm{u}}(\A_F)\oplus \A_F^2),  & 0 \leq i \leq \ell-1\\
I: \mathcal{S}(V_{i+1}^u(\A_F) )& \lto C^\infty(X_i^{u\circ}(\A_F)),  &0 \leq i \leq \ell-1
\end{align*}
and prove the following corollary of Theorem \ref{thm:main}:
\begin{cor} \label{cor:main} For $f \in \mathcal{S}(V_\ell'(\A_F) \oplus \A_F^2)$ one has that
\begin{align*}
&c_{\ell}(f)+\sum_{\xi \in X'^{\circ}_{\ell}(F)}I(f)(\xi)\\
&+\sum_{i=1}^{\ell-1}\left(c_i^{\mathrm{u}}(d_{\ell,i}^{\mathrm{u}}(f))+\sum_{\xi \in X^{\mathrm{u}\circ}_i(F)}I(d_{\ell,i}^{\mathrm{u}}(f))(\xi)\right)+ \kappa d_{\ell,0}(f)(0_{V_0},0,0)\\&
=c_\ell(\mathcal{F}_{X_\ell'}(f))+\sum_{\xi \in X'^\circ_\ell(F)}I(\mathcal{F}_{X_\ell'}(f))(\xi)\\&+
\sum_{i=1}^{\ell-1} \left(c_i^{\mathrm{u}}(d^{\mathrm{u}}_{\ell,i}(\mathcal{F}_{X_{\ell}'} (f)))+\sum_{\xi \in X^{\mathrm{u}\circ}_i(F)}I(d_{\ell,i}^{\mathrm{u}}(\mathcal{F}_{X_{\ell}'}(f)))(\xi)\right)+\kappa d_{\ell,0}^{\mathrm{u}}(\mathcal{F}_{X_{\ell}'}(f))(0_{V_0},0,0).
\end{align*}
\end{cor}
We do not know how to prove an analogue of Corollary \ref{cor:main} without fixing the subspace $V_0.$  When the quadratic form $Q_{\ell}'$ is split we have $V_0=\{0\}$ so this is not an issue.

\subsection{Sketch of the proof of Theorem \ref{thm:main}} \label{ssec:minimal}

In order to prove Theorem \ref{thm:main}, for $f \in \mathcal{S}(V_{\ell}(\A_F) \oplus \A_F^2)$  
we compute
\begin{align} \label{Theta:int}
\int_{[\SL_2]} \Theta_f^T(g)dg
\end{align}
where $\Theta_f$ is the usual $\Theta$-function and the superscript $T$ denotes the usual truncation operator employed by Arthur.    
\begin{thm} \label{thm:minimal} There is a polynomial $p_f(y_1,y_2) \in \CC[y_1,y_2]$ such that 
$$
\lim_{T \to \infty}\left(\int_{[\SL_2]}\Theta_f^T(g)dg-p_f(T,e^T)\right)=0.
$$
The constant term of $p_f(y_1,y_2)$ is
\begin{align*}
\sum_{i=1}^\ell\left(c_i(d_{\ell,i}(\mathcal{F}_2(f)))+\sum_{\xi \in X^{\circ}_i(F)}I(d_{\ell,i}(\mathcal{F}_2(f)))(\xi)\right)+\kappa d_{\ell,0}(\mathcal{F}_2(f))(0_{V_0},0,0).
\end{align*}
\end{thm}

\noindent Theorem \ref{thm:main} then follows upon observing that $\Theta_{\mathcal{F}_2^{-1}( \mathcal{F}_{X}(f))}(g)=\Theta_{\mathcal{F}_2(f)}(g)$ (see \S \ref{sec:proof} for more details and the proof of Theorem \ref{thm:minimal}).

We now explain a representation-theoretic interpretation of Theorem \ref{thm:minimal}.  For each $i$ let $\mathrm{O}_{V_i}$ (resp.~$\mathrm{GO}_{V_i}$) be the orthogonal (resp.~orthogonal similitude) group of $V_i$.
The partial Fourier transform $\mathcal{F}_2:\mathcal{S}(V_{i}(\A_F) \oplus \A_F^2) \to \mathcal{S}(V_{i}(\A_F) \oplus \A_F^2)$ induced by the local partial Fourier transforms in \eqref{F2}  intertwines the action of $\rho_{i+1}$ and $r_i$ (see Lemma \ref{lem:F2:equiv}).  Hence it yields a $\CC$-linear isomorphism
$$
\mathcal{F}_2:\mathcal{S}(V_{i}(\A_F) \oplus \A_F^2)_{\rho_{i+1}(\SL_2(\A_F) )} \tilde{\lto} \mathcal{S}(V_{i}(\A_F) \oplus \A_F^2)_{r_{i}(\SL_2(\A_F))}
$$
that is equivariant with respect to the action of $\mathrm{O}_{V_{i}}(\A_F)$ (embedded in $\mathrm{O}_{V_{i+1}}(\A_F))$ in the obvious manner).  The left hand side admits an obvious action of $\mathrm{GO}_{V_{i+1}}(\A_F)$ and we can define a representation 
$$
\sigma_i: \mathrm{GO}_{V_{i+1}}(\A_F) \times \mathcal{S}(V_{i}(\A_F) \oplus \A_F^2)_{r_{i}(\SL_2(\A_F)) } \lto \mathcal{S}(V_{i}(\A_F) \oplus \A_F^2)_{r_{i}(\SL_2(\A_F)) }
$$ 
by transport of structure.  We will make the action explicit in Proposition \ref{prop:action} below.  The representation $\sigma_i$ is the big theta lift of the trivial representation of $\SL_2(\A_F)$. 

 Theorem \ref{thm:minimal} amounts to an explicit automorphic realization of $\sigma_i$.  The realization is clearly similar in spirit to the Schr\"odinger model of the metaplectic representation of the two-fold cover of a symplectic group.  In the Archimedean case this analogy is discussed in detail in \cite{Kobayashi:Mano}.  
  We hope that having such an explicit model will aid in applications analogous to the many uses of the Schr\"odinger model.  In particular the explicit geometric description of the model may aid in its use in unfolding arguments that are so crucial in the theory of integral representations of $L$-functions.

We hasten to point out that it is already known that $\sigma_i$ is automorphic.
In the case where $V_0=\{0\}$ it is discussed in detail in \cite{GRS:theta}, which contains a wealth of information about the representation and analogues of it for exceptional groups.  In ibid., the theta lift is realized as a residue of an Eisenstein series.  It is then related to the theta lift of the trivial representation of $\SL_2(\A_F)$ to $\mathrm{O}_{V_i}(\A_F)$ in the special case where $F$ is totally real using Kudla and Rallis' regularized Siegel-Weil identity \cite{Kudla:Rallis:Reg:SW}.  As far as the author knows, this perspective does not lead to a simpler proof of Theorem \ref{thm:minimal} (and hence Theorem \ref{thm:main}) valid for all test functions $f.$  This is one of the main features of the current paper.  For the purpose of comparing our work with \cite{GRS:theta} it is also worth mentioning that we make no use of the theory of Eisenstein series on groups of absolute rank bigger than $1$.  
There is another closely related automorphic realization that was studied by Kazhdan and Polishchuk in \cite{Kazhdan:Polishchuk}.  However, this automorphic realization was not completely determined.  In particular it is not clear how to extract the explicit formulae of Theorem \ref{thm:main} and Theorem \ref{thm:minimal} from \cite{Kazhdan:Polishchuk}.  

We now outline the contents of this paper.   We set notational and measure conventions in \S \ref{sec:prelim}.    In \S \ref{sec:loc} we define local analogues of the operators $I$ and prove several useful properties of them.  The proof of Theorem \ref{thm:minimal} is by induction on $i$.  In \S \ref{sec:isotropic}  the computation of $\int_{[\SL_2]}\Theta_f^T(g)dg$ is inductively reduced to the case where $V_i=V_0$, that is, to the anisotropic case.  This case requires one additional idea and is treated in \S \ref{sec:anisotropic}.  We put the results of \S \ref{sec:isotropic} and \S \ref{sec:anisotropic} together in \S \ref{sec:proof} to prove Theorem  \ref{thm:minimal}.  We then deduce Theorem \ref{thm:main}.  In \S \ref{sec:invariance} we explain how to modify the identity of Theorem \ref{thm:main} so that all terms depend only on the image of $f$ in $\mathcal{S}(X_{\ell}(\A_F)).$
The proof of Corollary \ref{cor:main} is contained in \S \ref{sec:canon}.  In \S \ref{sec:circle:method} we indicate the relationship of our results with classical questions related to the circle method.

\section*{Acknowledgements}

The author thanks D.~Kazhdan for suggesting studying the Poisson summation formula for quadrics and for many useful conversations.  In particular he emphasized to the author the importance of establishing that the linear forms in the Poisson summation formulae are invariant under the action of $\SL_2(\A_F).$   A.~Pollack helped with references, Y.~Sakellaridis explained how to think about coinvariants, and
G.~Savin answered questions about the minimal representation, and H.~Yao pointed out a misprint in the definition of the minimal representation.  C-H. Hsu pointed out many typos in an earlier manuscript and was kind enough to write the companion paper \cite{Hsu:Companion} elucidating the local theory.   The author also thanks H.~Hahn for her help with editing and her constant encouragement.  Finally, the author acknowledges comments by the referee that greatly improved the exposition.  In particular, it was the referee's suggestion that \cite{Hsu:Companion} be developed into an independent work.  It was originally an appendix to this paper.

\section{Preliminaries}
\label{sec:prelim}
\subsection{Groups} \label{ssec:groups}

We equip $V_i:=V_0 \oplus \GG_a^{2i}$ with the quadratic form 
\begin{align} \label{Qi}
Q_i(x):=\frac{x^tJ_ix}{2}, \quad  J_i=\begin{psmatrix} J_0 & & & \\ & J & & \\ & &  \ddots & \\ & & & J \end{psmatrix}
\end{align}
where $J_0$ is the matrix of $Q_{0}$ and 
$J:=\begin{psmatrix} & 1 \\ 1 & \end{psmatrix},$ with $i$ copies of $J.$  
Let
\begin{align} \label{pairing}
\langle v_1,v_2\rangle_i:=v_1^tJ_iv_2
\end{align}
be the pairing attached to the quadratic form $Q_i.$ 
 For $i>i' \geq 0$ we identify $V_{i'}$ with the subspace 
$$
V_{i'} \oplus \{0\}^{2i-2i'} <V_i.
$$
In particular, we have an obvious isomorphism
\begin{align} \label{obvi}
    V_i \oplus \GG_a^2 \tilde{\lto} V_{i+1}.
\end{align}
We let 
\begin{align} \label{Vicirc}
    V_i^{\circ}:=V_i-\{0\}.
\end{align}

As in the introduction $\mathrm{GO}_{V_i}$ is the similitude group of $(V_i,Q_i)$.  Let
\begin{align} \label{lambda}
\lambda:\mathrm{GO}_{V_i} \lto \GG_m
\end{align}
be the similitude norm.  We identify $\mathrm{GO}_{V_i}$ with a subgroup of $\mathrm{GO}_{V_{i+1}}$ via the embedding given on points in an $F$-algebra $R$ by 
\begin{align} \label{GOvi:embed}\begin{split}
\mathrm{GO}_{V_i}(R) &\lto \mathrm{GO}_{V_{i+1}}(R)\\
h &\longmapsto \begin{psmatrix} h & & \\ & \lambda(h) & \\ & & 1 \end{psmatrix}.\end{split}
\end{align}

For $x \in R^{\dim V_i}$ (viewed as a column vector) we let 
\begin{align} \label{ux}
u(x):=\begin{psmatrix} I_{V_i} & x& \\ & 1 & \\-x^tJ_i  & -Q_i(x) & 1\end{psmatrix}
\end{align}
and set
\begin{align} \label{Ni}
N_{i+1}(R):=\{u(x): x \in R^{\dim V_i} \}.
\end{align}
This is the unipotent radical of a maximal parabolic subgroup of (the neutral component of) $\mathrm{GO}_{V_{i+1}}$.  

\subsection{The Weil representation} \label{ssec:WR}
We define the semidirect product $\mathrm{SL}_2 \rtimes \mathrm{GO}_{V_i}$ by stipulating that 
 $$
 (g \rtimes h)(g' \rtimes h)=g\begin{psmatrix} 1 & \\ & \lambda(h)\end{psmatrix} g' \begin{psmatrix} 1 & \\ & \lambda(h) \end{psmatrix}^{-1} \rtimes h h'.
 $$

 Let $v$ be a place of $F$ which we omit from notation, writing $F:=F_v.$ 
Let 
\begin{align} \label{rhoi:loc}
\rho_i:=\rho_{i,\psi}:\mathrm{SL}_2(F) \rtimes \mathrm{GO}_{V_i}(F) \times \mathcal{S}(V_i(F)) \lto \mathcal{S}(V_i(F))
\end{align}
be the Weil representation. 
  Thus $\rho_i$ is a local factor of the adelic Weil representation in \eqref{rhoi}. 

 For the convenience of the reader, we recall the definition of $\rho_i.$  
For $(h,f) \in \mathrm{GO}_{V_i}(F) \times \mathcal{S}(V_i(F))$ one has 
$\rho_i(I_2 \rtimes h)f(v)=f(h^{-1}v).$  
We usually work with the restriction of the Weil representation to $\SL_2(F) \rtimes I_{V_i},$ and hence suppress the similitude group $\mathrm{GO}_{V_i}(F)$ from notation.  

To recall the action of $\mathrm{SL}_2(F)$ we set some notation.
Let 
\begin{align} \label{gamma}
\gamma:=\gamma(Q_i)
\end{align}
be the Weil index (it is independent of $i$) and let
\begin{align} \label{chi}
\chi(a):=(a,(-1)^{\dim V_i/2}\det J_i)
\end{align}
be the usual character attached to the quadratic form (it is independent of $i$).  Here the right hand side of \eqref{chi} is the usual quadratic Hilbert symbol.
The Weil representation is characterized uniquely by the following requirements:
\begin{enumerate}
\item $\rho_i\begin{psmatrix} & 1\\ -1 & \end{psmatrix}f(v)=\gamma\int_{V_i(F)}f(t)\psi(v^tJ_it)dt$
\item $\rho_i\begin{psmatrix} 1 & t\\ & 1 \end{psmatrix}f(v)=\psi(tQ_i(v))f(v) \textrm{ for }t \in F$ \label{20}
\item $\rho\begin{psmatrix} a & \\ & a^{-1}\end{psmatrix}f(v)=\chi(a)|a|^{\dim_F V_i/2}f(av)$ for $a \in F^\times.$
\end{enumerate}
Here $\langle\,,\,\rangle_i$ is the pairing attached to $Q_i$ as in \eqref{pairing}.
It is not obvious that the Weil representation is extends to the semidirect product $\SL_2(F) \rtimes \mathrm{GO}_{V_i}(F);$ see \cite[\S 3.1]{Getz:Liu:Triple} for references.

\subsection{Coinvariants} \label{ssec:coin}
We revert to global notation; thus $F$ is a number field.
Let 
\begin{align} \label{ri:loc}
r_i:=\rho_i \otimes L^\vee:\SL_2(F_v) \times \mathcal{S}(V_i(F_v) \oplus F_v^2) \lto \mathcal{S}(V_i(F_v) \oplus F_v^2).
\end{align}
This is the local analogue of \eqref{ri}.

Temporarily let 
$$
W_v:=\big\langle f-r_i(g)f:(g,f) \in \SL_2(F_v) \times \mathcal{S}(V_{i}(F_v) \oplus F^2_v)  \big\rangle
$$
where the brackets denote the $\CC$-span.  
If $v$ is non-Archimedean set
\begin{align} \label{na:case}
\mathcal{S}(V_{i}(F_v)\oplus F^2_v)_{r_i(\SL_2(F_v))}:=\mathcal{S}(V_{i}(F_v) \oplus F^2_v)/W_v
\end{align}
If $S$ is a set of Archimedean places let 
\begin{align} \label{arch:case}
\mathcal{S}(V_{i}(F_S) \oplus F^2_S)_{r_i(\SL_2(F_S) )}:=\mathcal{S}(V_{i}(F_S) \oplus F^2_S)/\bar{\otimes_{v \in S}W_v}
\end{align}
where $\bar{\otimes_{v \in S}W_v}$ is the closure of $\otimes_{v \in S}W_v$ in the usual Fr\'echet space topology on $\mathcal{S}(V_{i}(F_S) \oplus F^2_S)$.  Let $\infty$ denote the set of infinite places of $F.$  We then set
\begin{align*}
&\mathcal{S}(V_i(\A_F) \oplus \A_F^2)_{r_i(\SL_2(\A_F))}\\&:=\mathcal{S}(V_i(F_\infty) \oplus F_\infty^2)_{r_i(\SL_2(F_\infty))} \otimes \mathcal{S}(V_i(\A_F^\infty) \oplus (\A_F^\infty)^2)_{r_i(\SL_2(\A_F^\infty))}
\end{align*}
where 
$\mathcal{S}(V_i(\A_F^\infty) \oplus (\A_F^\infty)^2)_{r_i(\SL_2(F_v))}=\otimes_{v \nmid \infty}'\mathcal{S}(V_i(F_v) \oplus F_v^2)_{r_i(\SL_2(F_v))}.$  Here the restricted direct product is with respect to the image of $\one_{V_i(\OO_{F_v}) \oplus \OO_{F_v}^2}$ in $\mathcal{S}(V_i(F_v) \oplus F_v^2)_{r_i(\SL_2(F_v))}.$

\subsection{Normalization of measures and the Harish-Chandra map}
\label{ssec:measures}
Fix a nontrivial character $\psi:F \backslash \A_F \to \CC^\times$. 
For each place $v$ of $F$ we normalize the Haar measure on $F_v$ so that the Fourier transform with respect to $\psi_v$ is self-dual.  Then the induced measure on $\A_F$ gives $F \backslash \A_F$ measure $1$ \cite[\S VII.2]{Weil:Basic:NT}.  For each place $v$ of $F$ we give $F_v^\times$ the measure 
$$
d^\times x_v:=\zeta_v(1)|x|_v^{-1}dx_v.
$$
The $d^\times x_v$ induce a measure on $\A_F^\times.$
If $F_v$ is unramified over its prime field and $\psi_v$ is unramified then $dx_{v}(\OO_v)=1$, where $\OO_v$ is the ring of integers of $F_v$.  Let $\infty$ denote the set of infinite places of $F$ and let $F_\infty:=\prod_{v|\infty}F_v.$

As usual, let 
\begin{align} \label{A}
A_{\GG_m} \leq F_\infty^\times
\end{align}
be the diagonal copy of $\RR_{>0}$ and let
\begin{align} \label{1:group}
    (\A_F^\times)^1:=\{x \in \A_F^\times:|x|=1\}.
\end{align}
We choose the Haar measure on $A_{\GG_m}$ so that the isomorphism $|\cdot|:A_{\GG_m} \tilde{\to} \RR_{>0}$ is measure preserving and then endow  $(\A_F^\times)^1$ with the unique Haar measure such that the  canonical isomorphism
$$
A_{\GG_m} \times (\A_F^\times)^1 \tilde{\lto}\A_F^\times
$$
is measure-preserving.

Let $T_2 < B < \SL_2$ be the maximal torus of diagonal matrices and the Borel subgroup of upper triangular matrices, respectively.   Let $N$ be the unipotent radical of $B.$ 
Let $K=\prod_{v \infty}K_v$ be the usual maximal compact subgroup; thus $K_v$ is $\mathrm{SO}_2(\RR)$ (resp.~$\mathrm{SU}_2(\RR)$,  $\mathrm{SL}_2(\OO_v)$) if $v$ is real (resp.~ complex, finite).  By the Iwasawa decomposition
\begin{align}
\SL_2(\A_F)=N(\A_F) T_2(\A_F) K.
\end{align}
For each place $v$ of $F$ we give $\SL_2(F_v)$ the Haar measure 
$$
d\left(\begin{psmatrix} 1 & t\\ & 1\end{psmatrix} \begin{psmatrix} a & \\ & a^{-1}\end{psmatrix} k \right)=dt\frac{d^\times a }{|a|^2} dk
$$
where $(t,a,k) \in F_v \times F_v^\times \times K_v$ and $dk$ gives $K_v$ measure $1$.  This induces a measure on $\SL_2(\A_F).$

 For $g \in \SL_2(\A_F)$ write $g=n\begin{psmatrix}a & \\ & a^{-1} \end{psmatrix}k$ with $(n,a,k) \in N(\A_F) \times \A_F^\times \times K$ and define
\begin{align} \label{HB}
H_B(nak):=\log |a| \in \RR.
\end{align}
We also use the obvious local analogue of this notation.

If $G$ is an algebraic group over $F$ we let
$$
[G]:=G(F) \backslash G(\A_F).
$$
We let $\zeta(s)=\zeta_\infty(s)\zeta^\infty(s)$ denote the completed $L$-function.
In the following lemma the measure on $[\SL_2]$ is induced from the measure on $\SL_2(\A_F)$ fixed above.
\begin{lem}\label{lem:tama} 
We have $\mathrm{meas}([\SL_2])=|D|^{1/2}\zeta(2)$, where $D$ is the absolute discriminant of $F$.
\end{lem} 

\begin{proof}
It is well known that the Tamagawa number of $\SL_2$ is $1$ \cite[\S 5.3]{Platonov:Rapinchuk:AGNT}.  The measure we gave above on $\SL_2(\A_F)$ is
\begin{align*}
    |D|^{3/2}|D|^{-1/2}|D|^{-1/2}\cdot  \zeta(2)=|D|^{1/2}\zeta(2).
\end{align*}
times the Tamagawa measure by 
the computation in  \cite[\S 6]{Borel:3fold}.
\end{proof}

\subsection{Little-o notation}
Let $a,b:\RR_{>0} \to \CC$ be functions, possibly depending on some object $f.$ 
One writes
$$
a(T)=b(T)+o_f(1)
$$
if for all $\epsilon>0$ there is a $T_0 \in \RR_{>0},$ possibly depending on $\epsilon$ and $f,$ such that $|a(T)-b(T)|<\epsilon.$

\section{Local integrals} \label{sec:loc}

In this section we fix a place $v$ of $F$ and omit it from notation, writing $F:=F_v$.  Thus $F$ can be any local field of characteristic $0.$

Let $f \in \mathcal{S}( V_{i}(F) \oplus F^2)$. 
For $\xi \in X^{\circ}_i(F)$ let
\begin{align} \label{I:loc:def}
I(f)(\xi):=\int_{N(F) \backslash \SL_2(F)}r_i(g)f(\xi,0,1)d \dot{g}.
\end{align}
Here, as above, $N \leq \SL_2$ denotes the unipotent radical of the Borel subgroup of upper triangular matrices, and $r_i$ is defined as in \eqref{ri:loc}. This is the local analogue of the operator \eqref{I}.  

Let $V_i^\circ:=V_i-\{0\}$ as in \eqref{Vicirc}.  
When $F$ is non-Archimedean and 
$$
\xi=(\xi_1,\dots,\xi_{\dim V_i}) \in V_i(F)
$$ let
$$
|\xi|:=\max(|\xi_1|,\dots,|\xi_{\dim V_i}|).
$$
When $F$ is Archimedean choose a norm $|\cdot|:V_i(F) \to \RR_{\geq 0}.$  

\begin{lem} \label{lem:conv} The integral
\begin{align} \label{for:eps}
\int_{N(F) \backslash \SL_2(F)}|r_i(g)f\left(\xi,0,1\right)|d \dot{g}
\end{align}
is convergent for all $\xi \in V^\circ_i(F)$.  Let $\epsilon>0$ and $A \in \ZZ_{\geq 0}.$  The integral \eqref{for:eps} is bounded by a constant depending on  $f,$ $F,$ $A$ (and $\epsilon$ if $\dim V_i=4$) times 
$$
\max(|\xi|,1)^{-A}\begin{cases} \min(|\xi|,1)^{2-\dim V_i/2} &\textrm{ if }\dim V_i>4\\
\min(|\xi|,1)^{-\varepsilon} & \textrm{ if }\dim V_i=4\\
1 &\textrm{ if }\dim V_i=2.
\end{cases}
$$
In the non-Archimedean case, this can be strengthened to the assertion that the support of the integral is contained in the intersection of a compact subset of $V(F)$ with $V^{\circ}(F)$.  
\end{lem}

\begin{proof}
Decomposing the Haar measure on $\SL_2(F)$ using the Iwasawa decomposition we see that the integral in the lemma is equal to
\begin{align*}
\int_{F^\times \times K} &
\left|r_i\left(
 \begin{psmatrix} a & \\ & a^{-1}\end{psmatrix}k\right)f\right|\left(\xi,0,1
\right)\frac{d^\times a dk}{|a|^2}\\
=&\int_{F^\times \times K} \left|r_i( k)f\right|\left(
a\xi,0,a^{-1}\right)\frac{|a|^{\dim V_i/2}d^\times a dk}{|a|^2}\\
=&\int_{F^\times \times K} \left|r_i( k)f\right|\left( a^{-1}\xi,0, a\right)|a|^{2-\dim V_i/2}d^\times a dk.
\end{align*}
Here the powers of $a$ appear due to the fact that $r_i=\rho_i \otimes L^\vee$ (see \S \ref{ssec:WR} for the definition of the Weil representation $\rho_i$).

Now assume that $F$ is non-Archimedean.  Then the integral is compactly supported as a function of $\xi \in V_i(F).$  Moreover, it is bounded by a constant depending on $f$ times
\begin{align}
\int_{ |\xi| \ll_f |a| \ll_f 1} |a|^{2-\dim V_i/2}d^\times a.
\end{align}
This trivially yields the estimate in the lemma.

The Archimedean case is essentially the same, but requires a small computation which is contained in  \cite[Lemma 8.1]{Getz:Liu:Triple}.
\end{proof}
We point out that the integral in Lemma \ref{lem:conv} is well-defined for all $\xi \in V_i^\circ(F),$ not just $\xi \in X^{\circ}_i(F),$ due to the absolute values.  We require the bound for all $\xi \in V_i^\circ(F)$ to justify bringing certain integrals inside sums in Lemma \ref{lem:break} below.

  The integral $I$ defined in \eqref{I:loc:def} induces a morphism
$$
I:\mathcal{S}(V_{i}(F) \oplus F^2) \lto C^\infty(X^{\circ}_i(F))
$$
that factors through $\mathcal{S}(V_i(F) \oplus F^2)_{r_i(\SL_2(F))}$.  

We have two representations of $\mathrm{SL}_2(F) $ on $\mathcal{S}(V_{i}(F) \oplus F^2)$, namely $\rho_{i+1}(g )$ and $r_i(g)$ (see \eqref{rhoi:loc} and \eqref{ri:loc}).  We now relate these two actions.
 Define a transform 
\begin{align} \begin{split}
\mathcal{F}_2:\mathcal{S}(F^2) &\lto \mathcal{S}(F^2)\\
f &\longmapsto \left((u_1,u_2) \mapsto\int_{F} f(u_1,x)\psi(u_2x)dx \right)\end{split}
\end{align}
where $(u_1,u_2) \in F^2.$  
The subscript $2$ is a reminder that this is the Fourier transform in the second variable. Identifying $V_i(F) \oplus F^2=V_{i+1}(F)$  using the obvious isomorphism \eqref{obvi}, the Fourier transform extends to 
\begin{align} \label{F2}
\mathcal{F}_2:= 1_{\mathcal{S}(V_i(F))} \otimes \mathcal{F}_2:\mathcal{S}( V_{i+1}(F) ) \lto \mathcal{S}(V_{i}(F) \oplus F^2).
\end{align}
The canonical (right) action  $V_{i+1}(F) \times \mathrm{GO}_{V_{i+1}}(F) \to V_{i+1}(F)$ restricts to an action of $\mathrm{GO}_{V_i}(F)$ via the embedding $\mathrm{GO}_{V_i}(F) \to \mathrm{GO}_{V_{i+1}}(F)$ of \eqref{GOvi:embed}.  It then pulls back to an action of $\mathrm{GO}_{V_i}(F)$ on $V_i(F) \oplus F^2$ under \eqref{obvi}.  The map $\mathcal{F}_2$ intertwines the two actions of $\mathrm{GO}_{V_i}(F).$

We also have the following equivariance property:

\begin{lem} \label{lem:F2:equiv}
For $g \in \SL_2(F),$ one has that
$$
\mathcal{F}_2 \circ \rho_{i+1}(g )=r_i(g) \circ \mathcal{F}_2.
$$
\end{lem}

\begin{proof}
Let $f \in \mathcal{S}(V_i(F) \oplus F^2)$.
By the Bruhat decomposition, it suffices to check the identity for $g \in T_2(F)$, $g \in N(F)$ and $g=\begin{psmatrix} & 1 \\ -1 & \end{psmatrix}$. We use the formulae for the Weil representation recalled in \S \ref{ssec:WR} in the argument below.

 For $a \in F^\times,$ one has
\begin{align*}
\mathcal{F}_2 \circ \rho_{i+1}\begin{psmatrix} a & \\ & a^{-1}\end{psmatrix}f(\xi,u_1,u_2)&=\chi(a)|a|^{\dim V_i/2+1}\int_{F}f
\left(a\xi,a u_1 ,ax\right)\psi(u_2x)dx\\
&=r_i\begin{psmatrix} a & \\ & a^{-1}\end{psmatrix}\mathcal{F}_2(f)\left(\xi,u_1,u_2\right).
\end{align*}
For $ t\in F$ one has 
\begin{align*}
\mathcal{F}_2 \circ \rho_{i+1}\begin{psmatrix} 1 & t\\ & 1\end{psmatrix}f(\xi,u_1,u_2)&=\mathcal{F}_2(\psi(tQ_{i+1}(\cdot))f)\left(\xi,u_1,u_2\right)\\
&=\int_{F} f\left(\xi,u_1, x\right)\psi(tQ_i(\xi)+tu_1x)\psi(u_2x)dx\\
&=\psi(tQ_i(\xi))\mathcal{F}_2(f)\left(\xi,u_1,u_2+tu_1\right)\\
&=r_i\begin{psmatrix}1 & t\\ &1 \end{psmatrix}\mathcal{F}_2(f)\left(\xi,u_1,u_2\right).
\end{align*}
Moreover
\begin{align*}
&\mathcal{F}_2 \circ \rho_{i+1}\begin{psmatrix} & 1\\-1 &\end{psmatrix}f\left(\xi,u_1,u_2\right)\\&=\gamma \int_{F}\psi(u_2x) \Bigg(\int_{V_i(F) \times F^2} f\left( w,w_1,w_2\right)\psi(\langle \xi,w \rangle_{i}+u_1w_2+xw_1) dw_1dw_2dw\Bigg) dx\\
&=\gamma \int_{F \times V(F)} f\left( w,-u_2,w_2\right)\psi(\langle \xi,w \rangle_{i}+ u_1w_2) dw_2dw\\
&=r_i\begin{psmatrix} & 1\\-1 &\end{psmatrix}\circ \mathcal{F}_2(f)\left(\xi,u_1,u_2\right).
\end{align*}
\end{proof}

The transform $\mathcal{F}_2$ plays key role in the proof of Theorem \ref{thm:main}, see Lemma \ref{lem:break} below.  It is also important because it is what allows us to view $\mathcal{S}(X_i(F))$ as a representation of $\mathrm{GO}_{V_{i+1}}(F)$, not just $\mathrm{GO}_{V_i}(F),$ as we now explain.
By Lemma \ref{lem:F2:equiv} the isomorphism \eqref{F2} yields an isomorphism
$$
\mathcal{F}_2:\mathcal{S}(V_{i+1}(F))_{\rho_{i+1}(\SL_2(F))} \tilde{\lto} \mathcal{S}(V_{i}(F) \oplus F^2)_{r_i(\SL_2(F))}.
$$
We have an action
\begin{align} \label{L} \begin{split}
L: \mathrm{GO}_{V_{i+1}}(F) \times \mathcal{S}(V_{i+1}(F))_{\rho_{i+1}(\SL_2(F))} &\lto \mathcal{S}(V_{i+1}(F))_{\rho_{i+1}(\SL_2(F))}\\
(h,f) &\longmapsto (x \mapsto f(h^{-1}x)).
\end{split}
\end{align}
Thus we obtain an action of
$\mathrm{GO}_{V_{i+1}}(F)$ on $\mathcal{S}(V_i(F) \oplus F^2)$ by transport of structure: 
\begin{align} \label{sigi}
\sigma_i(h):=\mathcal{F}_2 \circ L(h) \circ  \mathcal{F}_2^{-1}.
\end{align}
The group $\SL_2(F) \rtimes \mathrm{GO}_{V_{i+1}}(F)$ acts on $\mathcal{S}(V_{i+1}(\A_F))$ via
  $$
  \rho_{i+1}(g \rtimes h)f:=\rho_{i+1}(g)(L(h)f)
  $$
(see \S \ref{ssec:WR}).
Thus $r_i$, originally defined as a representation of $\SL_2(F)$, extends to an action 
$$
r_i:\SL_2(F) \rtimes \mathrm{GO}_{V_{i+1}}(F) :\mathcal{S}(V_i(F) \oplus F^2) \lto \mathcal{S}(V_i(F) \oplus F^2)
$$
given by 
$$
  r_{i}(g \rtimes h)f:=r_{i}(g)(\sigma_i(h)f).
$$
This implies that $\sigma_i$ descends to an action
\begin{align}
\sigma_i:\mathrm{GO}_{V_{i+1}}(F)  \times \mathcal{S}(V_i(F) \oplus F^2)_{r_i(\mathrm{SL}_2(F))} \lto \mathcal{S}(V_i(F) \oplus F^2)_{r_i(\mathrm{SL}_2(F))}.
\end{align}

For $f \in \mathcal{S}(F^2)$ let
\begin{align} \label{Fwedge} \begin{split}
\mathcal{F}_\wedge(f)(v):&=\int_{F^2} f(w)\psi(w \wedge v)dw\\
&=\int_{F^2}f\begin{psmatrix} w_1 \\  w_2 \end{psmatrix}\psi(w_1v_2-w_2v_1)dw_1dw_2.
\end{split}
\end{align}
Thus $\mathcal{F}_\wedge$ is an $\SL_2(F)$-equivariant Fourier transform: 
\begin{align}
\mathcal{F}_\wedge:\mathcal{S}(F^2) \lto \mathcal{S}(F^2).
\end{align} 
We extend it to $\mathcal{S}(V_{i+1}(F))$ by setting $\mathcal{F}_\wedge:= 1_{\mathcal{S}(V_i(F))}\otimes\mathcal{F}_{\wedge} $.   It clearly descends to a linear isomorphism
\begin{align} \label{FXi}
\mathcal{F}_{X_i}:\mathcal{S}(V_i(F) \oplus F^2)_{r_i(\SL_2(F))} \lto \mathcal{S}(V_i(F) \oplus F^2)_{r_i(\mathrm{SL}_2(F))}.
\end{align}

It is useful to explicitly compute how the action of $\mathrm{GO}_{V_{i+1}}(F)$  interacts with the operator $I$.  
We use notation from \S \ref{ssec:groups}. 

\begin{prop} \label{prop:action}
Let $f \in \mathcal{S}(V_i(F) \oplus F^2)$ and $\xi \in X_i^\circ(F)$.    
For $a \in F^\times,$ $h \in \mathrm{GO}_{V_i}(F)$ and $x \in F^{\dim V_i}$ one has that 
\begin{align}
\label{1}I\left(\sigma_i(h)f \right)(\xi)&=|\lambda(h)|I(f)(h^{-1}\xi),\\
\label{2} I\left(\sigma_i\begin{psmatrix}  I_{V_i} & & \\ & a &  \\ & & a^{-1}\end{psmatrix}f \right) (\xi)&=\chi(a)|a|^{1-\dim V_i/2}
I(f)(a^{-1}\xi),\\ \label{3}
I\left(\sigma_i(u(x))f \right)(\xi)&=\overline{\psi}(\langle x,\xi\rangle_i)I(f)(\xi),\\\label{4}
\sigma_i\begin{psmatrix} I_{V_i}  & & \\  & & 1 \\ & 1 & \end{psmatrix}f&=\mathcal{F}_{X_i}(f).
\end{align}
\end{prop}

\noindent 
Thus the operator $I$ intertwines the representation $\mathcal{S}(V_i(F) \oplus F^2)_{r_i(\SL_2(F))}$ with the minimal representation of $\mathrm{SO}_{V_{i+1}}(F),$ realized on a suitable space of functions on $X_i(F)$ (see \cite{GurK:cone,Kazhdan:D4,Kobayashi:Mano,Kazhdan:Polishchuk} for more on the minimal representation).  This is made precise in \cite{Hsu:Companion}.
 
\begin{proof}
Changing variables $g \mapsto \begin{psmatrix} 1 & \\ &\lambda(h) \end{psmatrix} g \begin{psmatrix} 1 & \\ & \lambda(h)^{-1} \end{psmatrix}$ we have
\begin{align*}  \begin{split}
I(\sigma_i(h)f)(\xi)&=\int_{N(F) \backslash \SL_2(F)}r_i(g )\sigma_i(h)f(\xi,0,1)d\dot{g}\\
&=|\lambda(h)|\int_{N(F) \backslash \SL_2(F)}r_i(1 \rtimes h)r_i(g)f(\xi,0,1)d\dot{g}\\
&=|\lambda(h)|I(f)(h^{-1}\xi). \end{split}
\end{align*}
For $a \in F^\times$ one has 
\begin{align} \label{a:act} \begin{split}
\sigma_i&\begin{psmatrix}I_{V_i} &  & \\ & a & \\ & & a^{-1} \end{psmatrix}f(\xi,\xi_1',\xi_2')\\&=\int_F\left(\int_{F} f(\xi,a^{-1}\xi_1',x)\bar{\psi}(ayx)dx\right)\psi(\xi_2'y)dy\\
&=|a|^{-1}f(\xi,a^{-1}\xi'_1,a^{-1}\xi_2')\\&=\chi(a)|a|^{-\dim V_i/2-1}r_i\begin{psmatrix} a & \\ & a^{-1} \end{psmatrix} \sigma_i\begin{psmatrix} aI_{V_i}& & \\ & a^2 &  \\ & & 1 \end{psmatrix}f(\xi,\xi_1',\xi_2'). \end{split}
\end{align}
Thus
\begin{align*}
&I\left(\sigma_i\begin{psmatrix} I_{V_i}  & & \\ & a & \\ & & a^{-1} \end{psmatrix}f\right)(\xi)\\&=
\chi(a)|a|^{-\dim V_i/2-1}\int_{N(F) \backslash \SL_2(F)}r_i(g)r_i\begin{psmatrix} a & \\ & a^{-1} \end{psmatrix}\sigma_i\begin{psmatrix}aI_{V_i}& &  \\ & a^2 & \\ & & 1 \end{psmatrix}f(\xi,0,1)d\dot{g}\\
&=\chi(a)|a|^{-\dim V_i/2-1}\int_{N(F) \backslash \SL_2(F)}r_i(g)\sigma_i\begin{psmatrix} aI_{V_i}& &\\ & a^2 &  \\ & & 1 \end{psmatrix}f(\xi,0,1)d\dot{g}\\
&=\chi(a)|a|^{-\dim V_i/2-1}
I\left(\sigma_i\begin{psmatrix} aI_{V_i}& & \\ & a^2 &\\ & & 1 \end{psmatrix}f\right)(\xi).
\end{align*}
Here we have changed variables $g \mapsto g \begin{psmatrix} a & \\ & a^{-1} \end{psmatrix}^{-1}$.  We now apply \eqref{1} to see that the above is 
\begin{align}
\chi(a)|a|^{-\dim V_i/2+1}
I(f)(a^{-1}\xi).
\end{align}
For $f \in \mathcal{S}(V_i(F) \oplus F^2)$ one has 
\begin{align*}
L(u(x))\mathcal{F}_2^{-1}(f)(\xi,\xi_1',\xi'_2)&=
\int_{F}f(\xi-x\xi_1',\xi_1',v)\bar{\psi}((x^tJ_i\xi-Q_i(x)\xi'_1+\xi_2')v)dv.
\end{align*}
Thus
\begin{align*}
\sigma_i(u(x))f(\xi,\xi'_1,\xi'_2)
&=\int_F \left(\int_{F}f(\xi-x\xi_1',\xi_1',v)\bar{\psi}((x^tJ_i\xi-Q_i(x)\xi'_1+u)v)dv
\right) \psi(\xi_2'u)du\\
&=\int_F \left(\int_{F}f(\xi-x\xi_1',\xi_1',v)\bar{\psi}(uv)dv
\right) \psi(\xi_2'(-x^tJ_i\xi+Q_i(x)\xi'_1+u))du\\
&=f(\xi-x\xi_1',\xi_1',\xi'_2
) \overline{\psi}(\xi_2'(x^tJ_i\xi-Q_i(x)\xi'_1)).
\end{align*}
Since $\sigma_i(\mathrm{O}_{V_{i+1}}(F))$ commutes with $r_i(\SL_2(F))$ we deduce \eqref{3}.

We now prove \eqref{4}.
Temporarily let $f \in \mathcal{S}(F^2)$.
Letting $L(h)f(\xi)=f(h^{-1}\xi)$ as usual one has
\begin{align*}
L\begin{psmatrix} & 1 \\ 1 & \end{psmatrix} \mathcal{F}_2^{-1}(f)\begin{psmatrix} x \\ y \end{psmatrix}&=L\begin{psmatrix} & 1 \\ 1 & \end{psmatrix}\left( \begin{psmatrix}x \\ y \end{psmatrix} \mapsto \int_{F} f\begin{psmatrix} x\\ v \end{psmatrix}\psi(-yv)dv\right)\\
&=\int_{F} f\begin{psmatrix} y\\ v \end{psmatrix}\psi(-xv)dv.
\end{align*}
Applying $\mathcal{F}_2$ yields
\begin{align}
\int_{F^2} f\begin{psmatrix} u\\ v \end{psmatrix}\psi(-xv+yu)dudv=\mathcal{F}_\wedge( f)\begin{psmatrix} x\\ y \end{psmatrix}
\end{align}
and \eqref{4} follows.
\end{proof}

\begin{cor} \label{cor:FX:inv}
For $h \in \mathrm{GO}_{V_i}(F)$ one has
$$
I(\mathcal{F}_{X_i} \circ \sigma_i(h)f)(\xi)=|\lambda(h)|^{\dim V_i/2}I(\mathcal{F}_{X_i}(f))\left(\lambda(h)h^{-1}\xi \right).
$$
\end{cor}

\begin{proof}
One has that
$$
\begin{psmatrix} I_{V_i} & & \\ & & 1 \\ & 1 &  \end{psmatrix}\begin{psmatrix}  h& &  \\
& \lambda(h) & \\ & & 1 \end{psmatrix}=\begin{psmatrix} \lambda(h)^{-1}h & & \\ & \lambda(h)^{-1} & \\ & &1 \end{psmatrix}\lambda(h)I_{V_{i+1}}\begin{psmatrix}I_{V_i} & & \\ &  & 1\\ & 1 &\end{psmatrix}
$$
so by \eqref{1} and \eqref{4} we have
\begin{align*}
I(\mathcal{F}_{X_i}( \sigma_i(h)f))(\xi)=|\lambda(h)|^{-1}I(\sigma_i(\lambda(h)I_{V_{i+1}})\mathcal{F}_{X_i} (f))(\lambda(h)h^{-1}\xi).
\end{align*}

Now 
\begin{align*}
\lambda(h)I_{V_{i+1}}=\begin{psmatrix}  I_{V_i} & & \\ & \lambda(h)^{-1} &  \\ & & \lambda(h) \end{psmatrix}\begin{psmatrix} \lambda(h) & & \\ & \lambda(h)^2 & \\ & & 1\end{psmatrix}
\end{align*}
so by \eqref{1} and \eqref{2} we have 
\begin{align*}
&|\lambda(h)|^{-1}I(\sigma_i(\lambda(h)I_{V_{i+1}})\mathcal{F}_{X_i} (f))(\lambda(h)h^{-1}\xi)\\
&=\chi(\lambda(h))|\lambda(h)|^{\dim V_i/2-2}I\left(\sigma_i\begin{psmatrix} \lambda(h) & & \\ & \lambda(h)^2 & \\ & & 1\end{psmatrix}\mathcal{F}_{X_i}( f)\right)(\lambda(h)^2h^{-1}\xi)\\
&=\chi(\lambda(h))|\lambda(h)|^{\dim V_i/2}I(\mathcal{F}_{X_i} (f))(\lambda(h)h^{-1}\xi).
\end{align*}
Since $\chi$ is trivial on the image of the similitude norm \cite[Lemma 3.2]{Getz:Liu:Triple} we deduce the corollary.
\end{proof}

We will require another family of operators on $\mathcal{S}(V_i(F) \oplus F^2)$.  
For $f \in \mathcal{S}(V_{i}(F) \oplus F^2)$ and $s \in \CC$ define
\begin{align} \label{Z:loc:def}\begin{split}
Z_{r_i}(f,s):&=\int_{N(F) \backslash  \SL_2(F)} e^{H_B(g)(2-\dim V_i/2-s)}r_i(g)f\left(0_{V_i},0,1 \right) d \dot{g},\\
Z_{\rho_{i+1}}(f,s):&=\int_{N(F) \backslash \SL_2(F)} e^{H_B(g)(1-\dim V_i/2+s)}\rho_{i+1}(g)f\left(0_{V_i},0,1 \right) d \dot{g}. \end{split}
\end{align}
The integral $Z_{r_i}(f,s)$ is used in the definition of the functional $c_i$ in \eqref{ci:def} below.  
In Lemma \ref{lem:FE} we prove a functional equation for $Z_{r_i}(f,s)$ using the integrals $Z_{\rho_{i+1}}(f,s).$

\begin{lem} \label{lem:Tate}
The integral
$Z_{r_i}(f,s)$ is a Tate integral in the following sense:  
If 
$$
\Psi_{f}(x):=\int_{K} r_i(k)f\left(0,0,x\right) dk 
$$
then
\begin{align*}
Z_{r_i}(f,s)&=Z(\Psi_{f},\chi|\cdot|^{s}):=\int_{F^\times}\Psi_f(a)\chi(a)|a|^sd^\times a
\end{align*}
for $\mathrm{Re}(s)> 0.$
\end{lem}

\noindent In view of the lemma, the integral $Z_{r_i}(f,s)$ is meromorphic as a function of $s$.

\begin{proof}
For $\mathrm{Re}(s)>0,$ the Iwasawa decomposition implies that
\begin{align*}
Z_{r_i}(f,s)&=\int_{F^\times \times K}  \chi(a)|a|^{-s}(r_i(k)f)\left( 0_{V_i},0,a^{-1}\right) dk d^\times a\\
&=\int_{F^\times \times K}  \chi(a)|a|^{s}(r_i(k)f)\left(0_{V_i},0,a\right) dk d^\times a.
\end{align*}
\end{proof}

\begin{lem} \label{lem:Z}
Let $f \in \mathcal{S}(V_i(F) \oplus F^2)$.
For  $a \in F^\times,$ $h \in \mathrm{O}_{V_i}(F)$ and $x \in F^{\dim V_i}$ one has
\begin{align*} 
Z_{r_i}(\sigma_i(h)f,s)&=Z_{r_i}(f,s),\\
 Z_{r_i}\left(\sigma_i\begin{psmatrix} I_{V_i}& & \\ & a & \\ & &a^{-1}\end{psmatrix}f,s \right) &=\chi(a)|a|^{s-1}Z_{r_i}(f,s),\\  
Z_{r_i}\left(\sigma_i(u(x))f,s \right)&=Z_{r_i}(f,s).
\end{align*}
\end{lem}

\begin{proof}
The first assertion is clear.  
Similarly we compute
\begin{align*}
Z\left(\sigma_i\begin{psmatrix}  I_{V_i}  & & \\ & a &\\ & &a^{-1}\end{psmatrix}f,s \right)&=
\int_{F^\times \times K} \chi(b)|b|^{s}\left(\sigma_i\begin{psmatrix}  I_{V_i} & & \\ & a & \\ & &a^{-1}\end{psmatrix} r_i(k)f\right)(0,0,b)dk d^\times b.
\end{align*}
Using \eqref{a:act} this is 
\begin{align*}
\int_{F^\times \times K}& |b|^{s}\chi(b)|a|^{-1}\left( r_i(k)f\right)(0,0,a^{-1} b)dk d^\times b\\
=&|a|^{s-1}\chi(a)\int_{F^\times \times K} \chi(b)|b|^{s}\left( r_i(k)f\right)(0,0, b)d^\times b.
\end{align*}
The proof of the last formula is the same as the proof of \eqref{3} in Proposition \ref{prop:action}.
\end{proof}

 By a minor variant of the argument proving \cite[Lemma 5.3]{Getz:Hsu} one obtains the following lemma:

\begin{lem} \label{lem:contains} One has
$$
\mathcal{S}(V_i(F))\big|_{X^{\circ}_i(F)} <I(\mathcal{S}(V_i(F) \oplus F^2)).
$$ \qed
\end{lem}

For non-Archimedean $F$ with ring of integers $\OO$ we say that the image of $\one_{V_{i}(\OO) \oplus \OO^2}$ in $\mathcal{S}(V_i(F) \oplus F^2)_{r_i(\SL_2(F))}$ is the \textbf{basic function}.  Let $q$ be the order of the residue field of $\OO.$
Computing as in Lemma \ref{lem:conv} one obtains the following:

\begin{lem} \label{lem:unr:conv} Let $F$ be a non-Archimedean local field.  Assume that $\psi$ is unramified, that $F$ is unramified over its prime field, and that 
$J_i \in \GL_{V_i}(\OO)$.  For $\xi \in X_i^\circ(F)$ one has that
\begin{align*}
I(\one_{ V_{i}(\OO) \oplus \OO^2})(\xi):=\sum_{k=0}^\infty \chi(\varpi^k)q^{k(\dim V_i/2-2)}\one_{V_i(\OO)}\left(\frac{\xi}{\varpi^k} \right).
\end{align*} 
Moreover
\begin{align*}
\int_{N(F) \backslash \SL_2(F)}|r_i(g)\one_{V_{i+1}(\OO)}(\xi,0,1)|d \dot{g}
\end{align*}
is convergent for all $\xi \in V^\circ_i(F)$.  It is supported in  $V_i(\OO) \cap V^{\circ}_i(F)$ and is bounded by 
$$
\begin{cases} |\xi|^{2-\dim V_i/2}(\log_q |\xi|+1) & \textrm{ if } \dim V_i>4\\ \log_q |\xi|+1 &\textrm{ if }\dim V_i=4 \textrm{ or } 2.\end{cases}
$$
\qed
\end{lem}
\noindent If $F$ is non-Archimedean, unramified over its prime field, $\psi$ is unramified, and $J_i \in \GL_{V_i}(\OO)$ then it is clear that $\mathcal{F}_{X_i}$ preserves the basic function.

\section{Integrals of truncated isotropic theta functions} \label{sec:isotropic}

For $f \in \mathcal{S}( V_{i+1}(\A_F))$ with $i \geq 0$ and $g \in \SL_2(\A_F)$ let
\begin{align} \label{thetaf}
\Theta_f(g):=\sum_{\xi \in V_{i+1}(F)} \rho_{i+1}(g)f(\xi).
\end{align}
We refer to this as an \textbf{isotropic} theta function because $Q_{i+1}$ is isotropic.

For $T \in \RR_{>0}$ and suitable functions $\varphi$ on $[\SL_2]$ 
we then define
\begin{align} \label{varphiT}
\varphi^{T}(g):=\Lambda^T\varphi(g):=\varphi(g)-\sum_{\gamma \in B(F) \backslash \SL_2(F)} \one_{>T}(H_B(\gamma g))\int_{[N]}\varphi(n\gamma g)dn.
\end{align}
Here $\one_{>T}$ is the characteristic function of $\RR_{>T}$.  
This is the usual truncation in the special case of $\SL_2$ (in Arthur's notation, $\one_{>0}$ is $\widehat{\tau}_B$)  \cite[I.2.13]{MW:Spectral:Decomp:ES}.  
For fixed $g$ the sum over $\gamma$ in this expression is finite \cite[\S 6]{ArthurIntro}.

For $f \in \mathcal{S}(V_{i}(\A_F)\oplus \A_F^2)=\mathcal{S}(V_{i+1}(\A_F))$ and complex numbers $s$ with $\mathrm{Re}(s)>1$ define
\begin{align} \label{Z:def}\begin{split}
Z_{r_i}(f,s):&=\int_{N(\A_F) \backslash \SL_2(\A_F)} e^{H_B(g)(2-\dim V_i/2-s)}r_i(g)f\left(0,0,1 \right) d \dot{g},\\
Z_{\rho_{i+1}}(f,s):&=\int_{N(\A_F) \backslash \SL_2(\A_F)} e^{H_B(g)(1-\dim V_i/2+s)}\rho_{i+1}(g)f\left(0,0,1 \right) d \dot{g}. \end{split}
\end{align}
This is the global version of \eqref{Z:loc:def}.  The integral $Z_{r_i}(f,s)$ is a Tate integral by Lemma \ref{lem:Tate}.

 For each $i$ we define $\SL_2(\A_F)$-intertwining maps
\begin{align} \label{di0} \begin{split}
d_{i}:\mathcal{S}(V_{i}(\A_F) \oplus \A_F^2)& \lto \mathcal{S}(V_{i-1}(\A_F) \oplus \A_F^2)\\
f &\longmapsto \mathcal{F}_2(\xi \mapsto f(\xi,0,0)) \end{split}
\end{align}
where $\mathcal{F}_2:\mathcal{S}(V_i(\A_F)) \to\mathcal{S}(V_{i-1}(\A_F) \oplus \A_F^2)$ is the adelic analogue of the partial Fourier transform \eqref{F2}.

The main theorem of this section is the following:
\begin{thm} \label{thm:dim:red}  Let $f \in \mathcal{S}(V_{i}(\A_F) \oplus \A_F^2)$. One has 
\begin{align*}
\int_{[\SL_2]}|\Theta_f^T(g)|dg<\infty \quad \textrm{
and } \quad\sum_{\xi \in X_i^\circ(F)}|I(\mathcal{F}_2(f))(\xi)|<\infty.
\end{align*}
As $T \to \infty$
\begin{align*}
\int_{[\SL_2]}&\Theta_{f}^T(g)dg=\sum_{\xi \in X_i^{\circ}(F)}I(\mathcal{F}_2(f))(\xi)+\int_{[\SL_2]}\Theta_{d_i(f)}^T(g)dg\\&+\left(
  \sum_{s_i \in \left\{\frac{\dim V_i}{2}-1,\, \frac{\dim V_i}{2}-2,\,0\right\}}  \mathrm{Res}_{s=s_i}\frac{ e^{Ts}Z_{r_i}(\mathcal{F}_2(f),s+2-\frac{\dim V_i}{2})}{s}
\right)+o_f(1).
\end{align*}
\end{thm}
Here 
\begin{align*}
\Theta_{d_i(f)}(g):&=\sum_{\xi \in V_i(F)} \rho_i(g)d_i(f)(\xi)
\end{align*}
where  $\rho_i$ acts via its action on $\mathcal{S}(V_i(\A_F))$.  
 We give the proof at the end of this section.
By induction on $i$ this reduces the study of $\int_{[\SL_2]}\Theta_{f}^T(g)dg$ to a special case to be treated in \S \ref{sec:anisotropic}.

\begin{lem} \label{lem:FE}
One has
$$
Z_{r_i}(f,s)=Z_{\rho_{i+1}}(\mathcal{F}_2^{-1}(f),1-s)
$$
as meromorphic functions in $s$.
\end{lem}

\begin{proof}
Let
\begin{align*}
\Psi_f(x):= \int_K r_i(k)f(0, 0, x ) dk.
\end{align*}
One has
\begin{align*}
Z_{r_i}(f,s)
&=\int_{\A_F^\times \times K}  \chi(a)|a|^{s}
r_i(k)
f(0,0,a) dk d^\times a=Z(\Psi_f,\chi|\cdot|^s).
\end{align*}
Let $\widehat{\Psi}_f(x):=\int_{\A_F}\Psi_f(y)\bar{\psi}(xy)dy.$   Then we have
$$
Z(\Psi_f,\chi|\cdot|^s)=Z(\widehat{\Psi}_f,\chi|\cdot|^{1-s})
$$
by Tate's functional equation.  Here we are using the fact that $\chi=\chi^{-1}.$   Let us expand $Z(\widehat{\Psi}_f,\chi|\cdot|^s)$ for $\mathrm{Re}(s)$ sufficiently large.  It is equal to 
\begin{align*}
Z(\widehat{\Psi}_f,\chi|\cdot|^s)
&=\int_{\A_F^\times \times K}  \chi(a)|a|^{s}
\mathcal{F}_2^{-1}(r_i(k)f)(0,0,a)
 dk d^\times a\\
&=\int_{\A_F^\times \times K}  \chi(a)|a|^{-1-\dim V_i/2+s}
\rho_{i+1}\begin{psmatrix}a &  \\ & a^{-1} \end{psmatrix}\mathcal{F}_2^{-1} ( r_i(k)f)(0,0,1) dk d^\times a.
\end{align*}
By Lemma \ref{lem:F2:equiv} $ \mathcal{F}_2^{-1} \circ r_i= \rho_{i+1} \circ \mathcal{F}_2^{-1} $.   Thus the above is
\begin{align*}
&\int_{\A_F^\times \times K}  \chi(a)|a|^{-1-\dim V_i/2+s}
\rho_{i+1}\left(\begin{psmatrix}a &  \\ & a^{-1} \end{psmatrix} k\right)\mathcal{F}_2^{-1} (f)(0, 0, 1 ) dk d^\times a\\
 &=\int_{N(\A_F) \backslash \SL_2(\A_F)} e^{H_B(g)(1-\dim V_i/2+s)}\rho_{i+1}(g)\mathcal{F}_2^{-1} (f)(0,0,1)d \dot{g}\\
 &=Z_{\rho_{i+1}}(\mathcal{F}_2^{-1}(f),s).
\end{align*} 
\end{proof}

\begin{lem} \label{lem:break}
The integral $\int_{[\SL_2]}\Theta_{f}^T(g)dg$ is equal to 
\begin{align*} 
\begin{split}
\int_{[\SL_2]}\Theta^T_{d_i(f)}(g)dg+&\int_{N(\A_F) \backslash \SL_2(\A_F)} \Bigg(\one_{\leq T}(H_B(g))\sum_{ \xi \in  X_i(F)}r_i(ng)\mathcal{F}_2(f)(\xi,0,1)\\&-\one_{>T}(H_B(g)) 
\sum_{\xi \in V_i(F)}\int_{N(\A_F)} r_i(ng)\mathcal{F}_2(f)\left(\xi,1,0\right) dn
\Bigg)d\dot{g}.
\end{split}
\end{align*}
\end{lem}

\noindent Here is where we make key use of the transform $\mathcal{F}_2$.  It converts the integral of the truncated $\Theta$-function into an object that can be unfolded.
Before beginning the proof we make some remarks on convergence issues.   
The lemma reduces the absolute convergence of $\int_{[\SL_2]}\Theta_f^T(g)dg$ to the absolute convergence of $\int_{[\SL_2]}\Theta_{d_i(f)}^T(g)dg$ together with the absolute convergence of the other summand.  The absolute convergence of the other summand is checked in propositions \ref{prop:S1}, \ref{prop:S2}, and \ref{prop:S3}.  So by induction we are reduced to the absolute convergence statement in Lemma \ref{lem:0}.  Lemma \ref{lem:0} treats the case we excluded at the beginning of the current section when we assumed $i \geq 0.$  Altogether this justifies the manipulations below.  

\begin{proof}
By Lemma \ref{lem:F2:equiv} and Poisson summation we obtain
\begin{align} \label{F2:PS}
\Theta_f(g)=\sum_{\xi \in V_{i+1}(F)} r_i(g)\mathcal{F}_2(f)(\xi).
\end{align}
Thus
\begin{align*}
\int_{[\SL_2]}&\Theta_f^T(g)dg
=\int_{[\SL_2]} \Bigg(\sum_{(\xi,\xi'_1,\xi'_2) \in  V_i(F)\oplus F^2}r_i(g)\mathcal{F}_2(f)(\xi,\xi'_1,\xi_2') \\&- \sum_{\gamma \in B(F) \backslash \SL_2(F)}\one_{>T}(H_B(\gamma g)) \int_{[N]}\sum_{(\xi,\xi'_1,\xi'_2) \in  V(F) \oplus F^2}r_{i}(n\gamma g)\mathcal{F}_2(f)(\xi,\xi'_1,\xi'_2) dn\Bigg)dg.
\end{align*}

We separate the contributions of $(\xi'_1,\xi'_2)=(0,0)$ and $(\xi'_1,\xi'_2) \neq (0,0)$ to write this as the sum of 
\begin{align*}
&\int_{[\SL_2]} \Bigg(\sum_{\xi \in  V_i(F)}r_i(g)\mathcal{F}_2(f)(\xi,0,0) \\&- \sum_{\gamma \in B(F) \backslash \SL_2(F)}\one_{>T}(H_B(\gamma g)) \int_{[N]}\sum_{\xi \in V_i(F)}r_i(n\gamma g)\mathcal{F}_2(f)(\xi,0,0) dn\Bigg)dg\\
&=\int_{[\SL_2]}\Theta_{d_i(f)}^T(g)d\dot{g}
\end{align*}
and
\begin{align*} 
&\int_{[\SL_2]} \Bigg(\sum_{(\xi,\xi'_1,\xi'_2)}r_i(g)\mathcal{F}_2(f)(\xi,\xi_1',\xi_2') \\&- \sum_{\gamma \in B(F) \backslash \SL_2(F)}\one_{>T}(H_B(\gamma g)) \int_{[N]}\sum_{(\xi,\xi'_1,\xi'_2)}r_i(n\gamma g)\mathcal{F}_2(f)(\xi,\xi'_1,\xi'_2) dn\Bigg)dg 
\end{align*}
where the sums are over $(\xi,\xi'_1,\xi'_2) \in V_i(F)\oplus F^2$ such that $(\xi'_1,\xi'_2) \neq (0,0)$.  The latter expression is equal to 
\begin{align} \label{before:rearr} \begin{split}
&\int_{B(F) \backslash \SL_2(\A_F)} \Bigg(\sum_{ (\xi,\alpha) \in  V_i(F) \times F^\times}r_i(g)\mathcal{F}_2(f)(\xi,0,\alpha)\\& - \one_{>T}(H_B(g)) \int_{[N]}\sum_{(\xi,\xi'_1,\xi'_2)}r_{i}(n g)\mathcal{F}_2(f)(\xi,\xi'_1,\xi'_2) dn\Bigg)dg\\
&=\int_{N(F) \backslash \SL_2(\A_F)} \Bigg(\sum_{ \xi \in  V_i(F)}r_i(g)\mathcal{F}_2(f)(\xi,0,1)\\& - \one_{>T}(H_B(g)) \int_{[N]}\sum_{\xi \in V_i(F)}\sum_{(\xi_1',\xi'_2 )/\sim }r_i(n g)\mathcal{F}_2(f)(\xi,\xi'_1,\xi_2') dn\Bigg)dg.
\end{split}
\end{align}
The last sum is over $(\xi_1',\xi'_2) \in F^2-\{(0,0)\}$ up to equivalence, where $(\xi'_1,\xi'_2)$ is equivalent to $(\alpha \xi'_1,\alpha^{-1} \xi'_2)$ for all $\alpha \in F^\times$.

For each $\xi \in V_i(F)$ one has 
\begin{align*}
\int_{[N]}&\sum_{(\xi_1',\xi_2')/\sim }r_i(n g)\mathcal{F}_2(f)\left(\xi,\xi'_1,\xi'_2\right) dn\\
&=\int_{[N]}r_i(n g)\mathcal{F}_2(f)\left(\xi,0,1\right) dn+\int_{[N]} \sum_{\alpha \in F}r_i(ng)\mathcal{F}_2(f)\left( \xi,1,\alpha\right) dn.
\end{align*}
The left summand vanishes unless $\mathcal{Q}_i(\xi)=0$, in which case it is equal to $r_i( g)\mathcal{F}_2(f)\left(\xi,0,1\right)$.  Thus \eqref{before:rearr} is
\begin{align*}
    &\int_{N(F) \backslash \SL_2(\A_F)} \Bigg(\sum_{ \xi \in  V_i(F)-X_i(F)}r_i(g)\mathcal{F}_2(f)(\xi,0,1)+\one_{\leq T}(H_B(g))\sum_{\xi \in X_i(F)} r_i(g)\mathcal{F}_2(f)(\xi,0,1)   \\
&-\one_{>T}(H_B(g)) 
    \sum_{\xi \in V_i(F)} \int_{N(\A_F)}  r_i(ng)\mathcal{F}_2(f)\left( \xi,1,0\right) dn\Bigg)dg.
\end{align*}
Using lemmas \ref{lem:conv} and \ref{lem:unr:conv} we see that it is permissible to bring the integral over $[N]$ inside the sum over $\xi \in V_i(F) -X_i(F).$
Applying \eqref{20} in the definition of the Weil representation (see \S \ref{ssec:WR}),
we deduce that this contribution is zero.  The lemma follows.  
\end{proof}

We now break the second summand in Lemma \ref{lem:break} into three pieces that we compute in the following three propositions:

\begin{prop} \label{prop:S1}
The expression
$$
\int_{N(\A_F) \backslash \SL_2(\A_F)} \one_{\leq T}(H_B(g))\sum_{ \xi \in  X^\circ_i(F)}\left|r_i(g)\mathcal{F}_2(f)(\xi,0,1)\right| d\dot{g}
$$
converges and one has
$$
\lim_{T \to \infty} \int_{N(\A_F) \backslash \SL_2(\A_F)} \one_{\leq T}(H_B(g))\sum_{ \xi \in  X^\circ_i(F)}r_i(g)\mathcal{F}_2(f)(\xi,0,1) d\dot{g}=\sum_{\xi \in X_i^{\circ}(F)}I(\mathcal{F}_2(f))(\xi).
$$
The sum on the right is absolutely convergent.
\end{prop}

\begin{proof}
It is easy to see from Lemma \ref{lem:conv} and Lemma \ref{lem:unr:conv} that
\begin{align*}
&\sum_{\xi \in X_i^{\circ}(F)}\int_{N(\A_F) \backslash \SL_2(\A_F)}  |r_i(g)\mathcal{F}_2(f)|( \xi,0,1) d\dot{g}<\infty \quad \textrm{ and }\\
&\sum_{\xi \in X_i^\circ(F)}|I(\mathcal{F}_2(f))(\xi)|<\infty.
\end{align*}
The proposition thus follows from the dominated convergence theorem.
\end{proof}

\begin{prop} \label{prop:S2}
The integral
$$
\int_{N(\A_F) \backslash \SL_2(\A_F)} \one_{\leq T}(H_B(g))r_i(g)\mathcal{F}_2(f)(0_{V_i},0,1)
d\dot{g} 
$$
converges absolutely and 
is equal to  
\begin{align*}\sum_{s_i \in \left\{\frac{\dim V_i}{2}-1,\,\frac{\dim V_i}{2}-2,\,0\right\}}  \mathrm{Res}_{s=s_i}\frac{e^{Ts}Z_{r_i}(\mathcal{F}_2(f),s+2-\frac{\dim V_i}{2})}{s}+o_f(1)
\end{align*}
as $T \to \infty$.
\end{prop}

\begin{proof}
  We have
\begin{align} \label{before:Parseval} \begin{split}
&\int_{N(\A_F) \backslash \SL_2(\A_F)} \one_{\leq T}(H_B(g))
r_i(g)\mathcal{F}_2(f)(0_{V_i},0,1)d\dot g\\
&= \int_{\A_F^\times \times K} \one_{\leq T}(\log |a|)\chi(a)|a|^{\dim V_i/2}r_i(k)\mathcal{F}_2(f)(0_{V_i},0,a^{-1})\frac{dk d^\times a}{|a|^2}.  \end{split}
\end{align}
Note that this expression is absolutely convergent.  Indeed, the integral is supported in $|a|\leq e^T,$ and the integrand is rapidly decreasing as $|a| \to 0.$  

We wish to apply Mellin inversion to \eqref{before:Parseval}. Define $A_{\GG_m}$ as in \eqref{A}. For $\mathrm{Re}(s) >0$ we have
\begin{align}
    \int_{A_{\GG_m}}\one_{\leq T}(\log |a|) |a|^sd^\times a=\frac{e^{Ts}}{s}
\end{align}
and for $\mathrm{Re}(s)$ sufficiently small we have
\begin{align}
\int_{\A_F^\times \times K}\chi(a)|a|^{\dim V_i/2+s}r_i(k)\mathcal{F}_2(f)(0_{V_i},0,a^{-1})\frac{dk d^\times a}{|a|^2}=Z_{r_i}(\mathcal{F}_2(f),2-\tfrac{\dim V_i}{2}-s).
\end{align}
Hence \eqref{before:Parseval} is equal to 
\begin{align*}
    \frac{1}{2\pi i }\int_{i\RR+\sigma} e^{Ts}Z_{r_i}(\mathcal{F}_2(f),s+2-\tfrac{\dim V_i}{2})\frac{ds}{s}
\end{align*}
for $\sigma$ sufficiently large.

We now shift the contour to $\sigma$ very small to see that the integral in the proposition is equal to
\begin{align}
\sum_{s_i \in \left\{\frac{\dim V_i}{2}-1, \frac{\dim V_i}{2}-2,0\right\}}  \mathrm{Res}_{s=s_i}\frac{ e^{Ts}Z_{r_i}(\mathcal{F}_2(f),s+2-\tfrac{\dim V_i}{2})}{s}+o_f(1).
\end{align}
\end{proof}

\begin{prop} \label{prop:S3}
As $T \to \infty$ 
\begin{align*}
\int_{N(\A_F) \backslash \SL_2(\A_F)} \one_{>T}(H_B(g)) 
\int_{N(\A_F)}\sum_{\xi \in V_i(F)}|r_i(ng)\mathcal{F}_2(f)\left(\xi,1,0\right) |
d\dot{g}=o_f(1).
\end{align*}
\end{prop}

\begin{proof}
The integral in the proposition is equal to 
\begin{align*}
&\int_{\A_F^\times \times K}  \one_{>T}(\log |a|) 
\sum_{\xi \in V_i(F)}\int_{\A_F}|a|^{\dim V_i/2-2}
|r_{i}(k)\mathcal{F}_2(f)|(a\xi,a,a^{-1}x)dx
d^\times adk \\
&=\int_{\A_F^\times \times K} \one_{>T}(\log |a|) 
\sum_{\xi \in V_i(F)}\int_{\A_F}|a|^{\dim V_i/2-1}
|r_i(k)\mathcal{F}_2(f)|\left(
a\xi,a,x
\right) dx
d^\times a dk.
\end{align*}
Use the idelic norm $|\cdot|:A_{\GG_m} \tilde{\to}\RR_{>0}$ to identify $A_{\GG_m}$ and $\RR_{>0}.$  Then the above is 
\begin{align*}
&\int_{e^T}^\infty \int \Bigg( 
\sum_{\alpha \in F^\times} \sum_{\xi \in V_i(F)}\int_{\A_F}t^{\dim V_i/2-1}
|r_i(k)\mathcal{F}_2(f)|\left(
at\alpha \xi,at\alpha,x
\right) dx
\Bigg)d^\times a dk \frac{dt}{t}\\
&=\int_{e^T}^\infty \int \Bigg( 
\sum_{\alpha \in F^\times} \sum_{\xi \in V_i(F)}\int_{\A_F}t^{\dim V_i/2-1}
|r_i(k)\mathcal{F}_2(f)|\left(
at\xi,at\alpha,x
\right) dx
\Bigg)d^\times a dk \frac{dt}{t}
\end{align*}
where the middle integral is over $F^\times \backslash (\A_F^\times)^1 \times K.$
Since $F^\times \backslash (\A_F)^1 \times K$ is compact and the integral over $t$ is supported in $t>e^T$ we see that this expression converges absolutely.  It becomes smaller as $T$ becomes larger.  
\end{proof}

\begin{proof}[Proof of Theorem \ref{thm:dim:red}]
This is immediate from Lemma \ref{lem:break} and propositions \ref{prop:S1}, \ref{prop:S2}, and \ref{prop:S3}.
\end{proof} 

\section{Integrals of truncated anisotropic theta functions} \label{sec:anisotropic}

Assume for this section that $V_i:=V_{0}$ is an even dimensional vector space equipped with a nondegenerate anisotropic quadratic form $Q_0$.  Let $f \in \mathcal{S}(V_0(\A_F))$. 
 We refer to $\Theta_f(g)$ as an \textbf{anisotropic theta function}.
 We allow the special case where $V_0=\{0\}$.  In this case we define $\mathcal{S}(V_0(\A_F)):=\CC$ and the Weil representation is taken to be the trivial representation of $\SL_2(\A_F)$.  

Our aim is to compute $\int_{[\SL_2]}\Theta_f^T(g)dg$.  Since $Q_0$ is anisotropic we cannot reduce this computation to a smaller quadratic space as we did above.  Instead, we apply a variant of the classical Rankin-Selberg method.  Let $\Phi \in \mathcal{S}(\A_F^2).$   For $g \in \SL_2(\A_F)$ let
\begin{align}
\Phi_s(g):=\int_{\A_F^\times}\Phi(\begin{psmatrix} 0 & t \end{psmatrix}g)|t|^{2s}d^\times t.
\end{align}
Moreover let
\begin{align}\label{eq:sl2eis}
E(g,\Phi_s):=\sum_{\xi' \in F^2-\{0\} }\Phi_s(\xi' g).
\end{align}
Then, as is well-known, $E(g,\Phi_s)$ converges absolutely for $\mathrm{Re}(s)$ large enough and admits a meromorphic continuation to the plane.  Its residue at $s=1$ is 
$$
\frac{\widehat{\Phi}(0)}{2}:=\frac{1}{2}\int_{\A_F^2}\Phi(x,y)dxdy.
$$
In particular the residue is independent of $g.$  For all of this we refer the reader to \cite[\S 1]{JacquetZagier}.  Define 
\begin{align} \label{kappa0}
    \kappa:=\begin{cases}\mathrm{meas}([\SL_2]) &\textrm{ if }\dim V_0=0 \\
    0 &\textrm{ otherwise.}\end{cases}
\end{align}

\begin{lem} \label{lem:0} 
The integral $\int_{[\SL_2]}\Theta_f^T(g)dg$ converges absolutely.  If $\dim V_0=0$
$$
\lim_{T \to \infty}\int_{[\SL_2]}\Theta_f^T(g)dg=\kappa f(0).
$$
If $\dim V_0 \geq 4,$ then $\int_{[\SL_2]}\Theta_f^T(g)dg$ is a polynomial in $e^T$ and if $0<\dim V_0<4,$ it is a polynomial in $e^{-T}.$  If $\chi \neq 1$ then this polynomial vanishes identically.  If $\chi=1$ then the constant term of the polynomial is $0.$ 
\end{lem}

\begin{proof}
The lemma is easy to check when $\dim V_0=0.$  Thus for the remainder of the proof we assume $\dim V_0 >0.$
Assume that $\Phi \in \mathcal{S}(\A_F^2)$ and  $\Phi(\xi'k)=\Phi(\xi')$ for all $(\xi',k) \in  F^2 \times K.$  Assume moreover that $\widehat{\Phi}(0) \neq 0.$
Then by the comments before the statement of Lemma \ref{lem:0} we have 
$$
\frac{\widehat{\Phi}(0)}{2}\int_{[\SL_2]}\Theta_{f}^T(g)dg=\int_{[\SL_2]}\mathrm{Res}_{s=1}E(g,\Phi_s)\Theta_f^T(g)dg.
$$
The function $\Theta_f^T(g)$ is rapidly decreasing on $[\SL_2]$ (see \cite{Arthur:TFII}) and hence using  \cite[Lemma 4.2]{JacquetShalikaEPI} we deduce that the above is equal to 
\begin{align}
    \mathrm{Res}_{s=1}\int_{[\SL_2]}E(g,\Phi_s)\Theta_f^T(g)dg.
\end{align}
We also deduce that at $s=1$ the function $\int_{[\SL_2]}E(g,\Phi_s)\Theta_f^T(g)dg$ has at most a simple pole.
One has
\begin{align*}
\Theta_f^T(g)&=\Theta_f(g)-\sum_{\gamma \in B(F) \backslash \SL_2(F)} \one_{> T}(H_B(\gamma g))\int_{[N]} \Theta_f(n\gamma g)dn\\
&=\Theta_f(g)-\sum_{\gamma \in B(F) \backslash \SL_2(F)}\one_{>T}(H_B(\gamma g))\rho_0(\gamma g)f(0).
\end{align*}
Thus for $\mathrm{Re}(s)$ sufficiently large,
\begin{align*}
&\int_{[\SL_2]}E(g,\Phi_s)\Theta_f^T(g)dg\\
&=\int_{[\SL_2]} \sum_{\gamma' \in B(F) \backslash \SL_2(F)} \Phi_s(\gamma' g)\Big(\Theta_f(g)-\sum_{\gamma \in B(F) \backslash \SL_2(F)}\one_{>T}(H_B(\gamma g))\rho_0(\gamma g)f(0)\Big)dg\\
&=\int_{B(F) \backslash \SL_2(\A_F)}  \Phi_s(g)\Big(\Theta_f(g)-\sum_{\gamma \in B(F) \backslash \SL_2(F)}\one_{>T}(H_B(\gamma g))\rho_0(\gamma g)f(0)\Big)dg\\
&=\int_{T_2(F)N(\A_F) \backslash \SL_2(\A_F)} \Phi_s(g)\int_{[N]}\Big(\Theta_f(ng)-\sum_{\gamma \in B(F) \backslash \SL_2(F)}\one_{>T}(H_B(\gamma n g)\rho_0(\gamma n g)f(0))\Big)dn d\dot{g}.
\end{align*}
The manipulations here are justified since $E(g,\Phi_s)$ is of moderate growth for $\mathrm{Re}(s)$ large \cite[Lemma 4.2]{JacquetShalikaEPI} and $\Theta_f^T(g)$ is rapidly decreasing.

Using \eqref{20} in \S \ref{ssec:WR} the above is
 \begin{align} \label{before:2} \begin{split}
 \int_{T_2(F)N(\A_F) \backslash \SL_2(\A_F)} \Phi_s(g)&\Big(\rho_0(g)f(0)\one_{\leq T}(H_B(g))\\&-\int_{N(\A_F)}\one_{>T}(H_B(\begin{psmatrix} & 1 \\ -1 & \end{psmatrix} n g))\rho_0(\begin{psmatrix} & 1 \\ -1 & \end{psmatrix} n g)f(0)\Big)dn d\dot{g}. \end{split}
\end{align}
We break \eqref{before:2} into two summands.  The first is
\begin{align} \label{first:is}
\int_{T_2(F)N(\A_F) \backslash \SL_2(\A_F)} &\one_{\leq T}(H_B(g))\Phi_s(g)\rho_0(g)f(0)d\dot{g}.
\end{align}
For $(a,k) \in \A_F^\times \times K$ one has $\Phi_s(\begin{psmatrix}a & \\ & a^{-1} \end{psmatrix}k)=|a|^{2s}\Phi_s(I_2).$  Thus for $\mathrm{Re}(s)$ sufficiently large  \eqref{first:is} is
\begin{align*}
    &\Phi_s(I_2)\int_{[\GG_m] \times K} \one_{\leq T}(\log|a|)|a|^{2s+\dim V_0/2}\chi(a)\rho_0(k)f(0)\frac{d^\times a}{|a|^2}dk\\
   &= \Phi_s(I_2) \int_{F^\times \backslash (\A_F^\times)^1}\chi(u)d^\times u \int_K \rho_0(k)f(0)dk \int_{0}^{e^T}|r|^{2s+\dim V_0/2-3}dr.
\end{align*}

This vanishes unless $\chi=1.$  Assuming $\chi =1$ and $\mathrm{Re}(s)$ is sufficiently large we see that it is equal to
\begin{align*}
\Phi_s(I_2)\mathrm{meas}(F^\times \backslash (\A_F^\times)^1)\int_{K}\rho_0(k)f(0)dk \frac{e^{T(2s+\dim V_0/2-2)}}{2s+\frac{\dim V_0}{2}-2}.
\end{align*}
The residue of this expression at $s=1$ is zero for $\dim V_0>0.$

The second summand of \eqref{before:2} is 
\begin{align*}
&-\int_{T_2(F)N(\A_F) \backslash \SL_2(\A_F)} \Phi_s\left(g \right)\int_{N(\A_F)}\rho_0(\begin{psmatrix} & 1 \\ -1 & \end{psmatrix} n g)f(0)\one_{>T}(H_B(\begin{psmatrix} & 1 \\ -1 & \end{psmatrix}n g))dnd\dot{g}\\
&=-\int_{T_2(F) \backslash \SL_2(\A_F)} \Phi_s(g)\rho_0(\begin{psmatrix} & 1 \\ -1 & \end{psmatrix}  g)f(0)\one_{>T}(H_B(\begin{psmatrix} & 1 \\ -1 & \end{psmatrix} g))dg\\
&=-\int_{T_2(F) \backslash \SL_2(\A_F)} \Phi_s(\begin{psmatrix} & -1 \\ 1 & \end{psmatrix}g)\rho_0(g)f(0)\one_{>T}(H_B(g))dg\\
&= -\int_{K}\int_{[\GG_m]}\int_{\A_F}\int_{\A_F^\times}\Phi\begin{psmatrix} ta & txa^{-1} \end{psmatrix} |t|^{2s}d^\times t
|a|^{\dim V_0/2}\chi(a)
\rho_0(k)f(0)\one_{>T}(\log |a|)\frac{dxd^\times a dk}{|a|^2}.
\end{align*}
We change variables $t \mapsto ta^{-1}$ and then $x \mapsto t^{-1}xa^2$ to see that this is 
\begin{align*}
   & -\int_{K}\int_{[\GG_m]}\int_{\A_F}\int_{\A_F^\times}\Phi\begin{psmatrix} t & x \end{psmatrix} |t|^{2s-1}d^\times t
|a|^{\dim V_0/2-2s}\chi(a)
\rho_0(k)f(0)\one_{>T}(\log |a|)dxd^\times a dk\\
&= -\int_{K}\rho_0(k)f(0)dk
\int_{F^\times \backslash (\A_F^\times)^1} \chi(u)d^\times u
 \int_{e^T}^\infty r^{\dim V_0/2-2s-1}dr \int_{\A_F}\int_{\A_F^\times}\Phi\begin{psmatrix} t & x \end{psmatrix} |t|^{2s-1}d^\times t
dx.
\end{align*}
This vanishes unless $\chi$ is trivial, in which case it is equal to 
\begin{align} \label{before:res}\begin{split}
    -&\int_{K}\rho_0(k)f(0)dk\mathrm{meas}(F^\times \backslash (\A_F^\times)^1) \int_{e^T}^\infty r^{\dim V_0/2-2s-1}dr \int_{\A_F}\int_{\A_F^\times}\Phi\begin{psmatrix} t & x \end{psmatrix} |t|^{2s-1}d^\times t
dx \\
&=\int_{K}\rho_0(k)f(0)dk\mathrm{meas}(F^\times \backslash (\A_F^\times)^1) 
\frac{e^{T(\dim V_0/2-2s)}}{\frac{\dim V_0}{2}-2s}
 \int_{\A_F}\int_{\A_F^\times}\Phi\begin{psmatrix} t & x \end{psmatrix} |t|^{2s-1}d^\times t
dx. \end{split}
\end{align}
Assume for the moment that $\dim V_0=4.$  If $\int_{K}\rho_0(k)f(0)dk =0$ then this expression vanishes.  If $\int_{K}\rho_0(k)f(0)dk \neq 0$ then this expression has a pole of order $2$ at $s=1$ for suitably chosen $\Phi.$  This contradicts the fact that $\int_{[\SL_2]}E(g,\Phi_s)\Theta_f^T(g)dg$ has at most a simple pole at $s=1.$  Thus if $\dim V_0=4$ we are done.

Assume $\dim V_0 \neq 4.$  Then the pole of \eqref{before:res} at $s=1$ is simple with residue equal to 
$$
\mathrm{meas}(F^\times \backslash (\A_F^\times)^1) \int_K \rho_0(k)f(0)dk \frac{\widehat{\Phi}(0)e^{T(\dim V_0/2-2)}}{2(\frac{\dim V_0}{2}-2)}.
$$
This is a polynomial in $e^T$ when $\dim V_0>4$ and a polynomial in $e^{-T}$ when $0<\dim V_0<4.$  In either case the constant term is term $0.$
\end{proof}

\section{Proof of Theorem \ref{thm:main}}
\label{sec:proof}

For $i> 0$ we define a linear form
\begin{align} \label{ci:def} 
c_i:\mathcal{S}(V_i(\A_F) \oplus \A_F^2) \lto \CC
\end{align}
by
\begin{align*}
c_i(f):=\begin{cases} Z_{r_i}(f,2-\frac{\dim V_i}{2})& \textrm{ if } Z_{r_i}(f,s) \textrm{ is holomorphic at }2-\frac{\dim V_i}{2}\\
\lim_{s \to 0}\frac{d}{ds}(sZ_{r_i}(f,s+2-\frac{\dim V_i}{2})) &\textrm{ if }Z_{r_i}(f,s) \textrm{ has a pole at }2-\frac{\dim V_i}{2}.
\end{cases}
\end{align*}

\begin{lem} \label{lem:const}
There is a polynomial $p_{f,0}(y_1,y_2) \in \CC[y_1,y_2]$ such that 
\begin{align*}
\lim_{T \to \infty}\left(\sum_{s_i \in \left\{\frac{\dim V_i}{2}-1, \frac{\dim V_i}{2}-2,0\right\}}  \mathrm{Res}_{s=s_i}\frac{ e^{Ts}Z_{r_i}(f,s+2-\tfrac{\dim V_i}{2})}{s}-p_{f,0}(T,e^T)\right)=0.
\end{align*} 
 The constant term of $p_{f,0}(y_1,y_2)$  is $c_i(f)$.
\end{lem}

\begin{proof}
By Lemma \ref{lem:Tate} $Z_{r_i}(f,s)$ is a Tate integral.  Thus it is meromorphic, and in fact holomorphic apart from possible simple poles at $s\in \{0,1\}.$  The lemma follows.
\end{proof}

  For $i >i' \geq 0$ we let
\begin{align} \label{di}
d_{i,i'}=d_{i'+1} \circ \dots \circ d_{i-1} \circ d_{i}:\mathcal{S}(V_i(\A_F) \oplus \A_F^2) \lto \mathcal{S}(V_{i'}(\A_F) \oplus \A_F^2),
\end{align}
where $d_i$ is defined as in \eqref{di0}.
By convention, $d_{i,i}$ is the identity.  

\begin{proof}[Proof of Theorem \ref{thm:minimal}]
By Theorem \ref{thm:dim:red} we have 
\begin{align*}
\int_{[\SL_2]}&\Theta_f^T(g)dg=\sum_{\xi \in X_{\ell}^{\circ}(F)}I(\mathcal{F}_2(f))(\xi)+\int_{[\SL_2]}\Theta_{d_i(f)}^T(g)dg\\&+\left(
\sum_{s_\ell \in \left\{\frac{\dim V_\ell}{2}-1, \frac{\dim V_\ell}{2}-2,0\right\}}  \mathrm{Res}_{s=s_\ell}\frac{ e^{Ts}Z_{r_\ell}(\mathcal{F}_2(f),s+\tfrac{\dim V_\ell}{2}-2)}{s}
\right)+o_f(1).
\end{align*}
By induction on $\ell$ we obtain 
\begin{align*}
\int_{[\SL_2]}&\Theta_f^T(g)dg=\int_{[\SL_2]}\Theta_{d_{\ell,0}(\mathcal{F}_2(f))}^T(g)dg+\sum_{i=1}^\ell\Bigg(\sum_{\xi \in X_i^{\circ}(F)}I(d_i(\mathcal{F}_2(f)))(\xi)+\\&+\left(
\sum_{s_i \in \left\{\frac{\dim V_i}{2}-1, \frac{\dim V_i}{2}-2,0\right\}}  \mathrm{Res}_{s=s_i}\frac{ e^{Ts}Z_{r_i}(\mathcal{F}_2(f),s+2-\tfrac{\dim V_i}{2})}{s}
\right)\Bigg)+o_f(1).
\end{align*}
We now conclude using Lemma \ref{lem:0} and Lemma \ref{lem:const}.
\end{proof}

\begin{proof}[Proof of Theorem \ref{thm:main}]
By Poisson summation and the fact that $\mathcal{F}_{X_\ell}$ is $\SL_2(\A_F)$-invariant we have
$$
\sum_{\xi \in V_i(F) \oplus F^2} r_i(g)f(\xi)=\sum_{\xi \in V_i(F) \oplus F^2} \mathcal{F}_{X_\ell}(r_i(g)f)(\xi)=\sum_{\xi \in V_i(F) \oplus F^2} r_i(g)\mathcal{F}_{X_\ell}(f)(\xi).
$$
By \eqref{F2:PS} this implies we have 
$$
\Theta_{\mathcal{F}_2^{-1}(f)}(g)=\Theta_{\mathcal{F}_2^{-1}(\mathcal{F}_{X_\ell} (f))}(g).
$$ 
 Thus
$$
\int_{[\SL_2]}\Theta_{\mathcal{F}_2^{-1}(f)}^T(g)dg =\int_{[\SL_2]}\Theta_{\mathcal{F}_2^{-1}(\mathcal{F}_{X_\ell}(f))}^T(g)dg.
$$
We conclude using Theorem \ref{thm:minimal}. 
\end{proof}

In applications of Poisson summation the behavior of the functions involved under scaling plays a key role.  Since this takes some thought to work out we make it explicit:
\begin{cor} \label{cor:scaling} Assume $\ell>0$, and that either $\dim V_\ell>4$ or $\chi \neq 1$.  
For $a \in \A_F^\times$ and $f \in \mathcal{S}(V_{\ell}(\A_F) \oplus \A_F^2)$ one has 
\begin{align*}
&|a|^{1-\dim V_i/2}\chi(a)c_\ell(f)+\chi(a)|a|^{1-\dim V_i/2}\sum_{\xi \in X_{\ell}^{\circ}(F)}I(f)(a^{-1}\xi)\\&+
|a|^{-1}\sum_{i=1}^{\ell-1}\left( c_i(d_{\ell,i}(f))+\sum_{\xi \in X^{\circ}_i(F)}I(d_{\ell,i}(f))(\xi)\right)+|a|^{-1}\kappa d_{\ell,0}(f)(0_{V_0},0,0)\\=&
|a|^{\dim V_i/2-1}\chi(a)c_\ell(\mathcal{F}_{X_\ell}(f))+\chi(a)|a|^{\dim V_i/2-1}\sum_{\xi \in X_{\ell}^{\circ}(F)}I\left(\mathcal{F}_{X_\ell}( f)\right)(a\xi)\\&+|a|\sum_{i=1}^{\ell-1}  \left(c_i(d_{\ell,i}(\mathcal{F}_{X_\ell}(f)))+\sum_{\xi \in X^{\circ}_i(F)}I(d_{\ell,i}(\mathcal{F}_{X_\ell}(f)))(\xi)\right)+|a|\kappa d_{\ell,0}(\mathcal{F}_{X_\ell}(f))(0_{V_0},0,0).
\end{align*}
\end{cor}

\begin{proof}
Let $f \in \mathcal{S}(V_{\ell}(\A_F) \oplus \A_F^2)$.
By \eqref{2} of Proposition \ref{prop:action} we have
\begin{align} \label{1:LHS}
I\left(\sigma_\ell\begin{psmatrix} I_{V_\ell} & & \\ & a & \\ & & a^{-1} \end{psmatrix}f\right)(\xi)=\chi(a)|a|^{1-\dim V_\ell /2}I(f)(a^{-1}\xi)
\end{align}
and by \eqref{2} and \eqref{4} of Proposition \ref{prop:action} we have
\begin{align} \label{1:RHS} \begin{split}
I\left(\mathcal{F}_{X_\ell}\left(\sigma_\ell\begin{psmatrix}  I_{V_\ell}  & & \\ & a & \\ & & a^{-1}  \end{psmatrix} \right)f\right)(\xi)&=I\left(\sigma_\ell\left( \begin{psmatrix} I_{V_\ell}  & & \\ & & 1 \\ & 1& \end{psmatrix}\begin{psmatrix}  I_{V_\ell} & &\\ & a & \\ & & a^{-1}  \end{psmatrix} \right)f\right)(\xi)\\
&=I\left(\sigma_i\left( \begin{psmatrix} I_{V_\ell} & & \\ & a^{-1} & \\ & & a  \end{psmatrix}\begin{psmatrix}I_{V_\ell} & & \\  & & 1 \\ & 1 & \end{psmatrix} \right)f\right)(\xi)\\
&=I\left(\sigma_\ell\begin{psmatrix}  I_{V_\ell}& & \\& a^{-1} &\\ & & a  \end{psmatrix}\mathcal{F}_{X_\ell} (f)\right)(\xi)\\
&=\chi(a)|a|^{\dim V_\ell/2-1}I\left(\mathcal{F}_{X_\ell} (f)\right)(a\xi).
 \end{split}
\end{align}
Similarly by Lemma \ref{lem:Z}
\begin{align} \label{ci:scale} \begin{split}
c_\ell\left(\sigma_\ell\begin{psmatrix}I_{V_\ell} & & \\ & a & \\ & & a^{-1} \end{psmatrix}f\right)&=|a|^{1-\dim V_\ell/2}\chi(a)c_\ell(f),\\
c_\ell\left(\mathcal{F}_{X_\ell}\left(\sigma_\ell\begin{psmatrix}  I_{V_\ell} & &\\ & a &  \\ & & a^{-1}  \end{psmatrix} \right)f\right)&=|a|^{\dim V_\ell/2-1}\chi(a)c_\ell(\mathcal{F}_{X_\ell}(f)). \end{split}
\end{align}
On the other hand \eqref{a:act} implies
\begin{align} \label{d:scale}\begin{split}
d_{\ell}\left(\sigma_\ell\begin{psmatrix}  I_{V_\ell} & & \\ & a & \\ & & a^{-1} \end{psmatrix}f\right)&=|a|^{-1}d_{\ell}(f), \\ d_{\ell}\left(\mathcal{F}_{X_\ell}\left(\sigma_\ell\begin{psmatrix}  I_{V_\ell} & & \\ & a & \\ & & a^{-1} \end{psmatrix}\right)f\right)&=|a|d_{\ell}(\mathcal{F}_{X_\ell}(f)).\end{split}
\end{align}
Thus applying Theorem \ref{thm:main} to the function $\sigma_i\begin{psmatrix} I_{V_i}  & & \\ & a &\\ & & a^{-1} \end{psmatrix}f$ we arrive at the asserted identity.  
\end{proof}
\noindent There is an analogue of Corollary \ref{cor:scaling} that is valid in the case $\dim V_\ell=4$ and $\chi=1.$  We omit it because it is slightly cumbersome to state.

\section{Invariance and the Schwartz space  $\mathcal{S}(X_\ell(\A_F))$} \label{sec:invariance}

The summation formula in Theorem \ref{thm:main}
is phrased in terms of functionals on $\mathcal{S}(V_{\ell}(\A_F) \oplus \A_F^2).$  
Define
\begin{align}
\mathcal{S}(X_{i}(\A_F)):=\mathcal{S}(V_i(\A_F) \oplus \A_F^2)_{r_i(\SL_2(\A_F))}.
\end{align}
We have a commutative diagram
\begin{equation*}
\begin{tikzcd}
    \mathcal{S}(V_i(\A_F) \oplus \A_F^2) \arrow[rd,"I"] \arrow[r] &\mathcal{S}(X_i(\A_F)) \arrow[d]\\
    & C^\infty(X_i^\circ(\A_F))
\end{tikzcd}
\end{equation*}
where the horizontal arrow is the canonical surjection.  By a slight abuse of notation we continue to denote by $I$ the map on the right.  We have already explained in the introduction after \eqref{SX} how it is reasonable to regard $\mathcal{S}(X_i(\A_F))$ as the Schwartz space of $X_i(\A_F),$ despite the fact that it is only $I(\mathcal{S}(X_i(\A_F)))$ that is a  space of functions on $X_i^\infty(\A_F).$  To further justify this, in this section we reformulate Theorem \ref{thm:main} in terms of $\mathcal{S}(X_i(\A_F)).$  The main result is Corollary \ref{cor:main2} below.

For this section we take the convention that $c_{i}=0$ if $i \leq 0$.
\begin{thm} \label{thm:inv}
Assuming $\dim V_i \not \in \{2,4\},$ the functionals $c_i$ are  $r_i(\SL_2(\A_F))$-invariant.  If $\dim V_i=4$ then 
\begin{align*}
\mathcal{S}(V_i(\A_F) \oplus \A_F^2)& \lto \CC\\
f &\longmapsto c_i(f)+c_{i-1}(d_i(f))
\end{align*}
is $r_i(\SL_2(\A_F))$-invariant.  
\end{thm}

We first prove a special case of Theorem \ref{thm:inv}:
\begin{lem} \label{lem:EZ:case}
If $\dim V_i \not \in \{2,4\}$ then $c_i$ is $r_i(\SL_2(\A_F))$-invariant. 
\end{lem}
 
\begin{proof}
By Lemma \ref{lem:F2:equiv} and Lemma \ref{lem:FE} it suffices to show that 
\begin{align*}
\mathcal{S}(V_{i+1}(\A_F)) &\lto \CC\\
f &\longmapsto Z_{\rho_{i+1}}(f,\tfrac{\dim V_i}{2}-1)
\end{align*}
is invariant under the action of $\rho_{i+1}(\SL_2(\A_F)).$  But this is clear.  
\end{proof}

\begin{lem} \label{lem:invariance} 
Assume $\dim V_i=4$  and that $f \in \mathcal{S}(V_{i+1}(\A_F)\oplus \A_F^2).$
The difference 
\begin{align*}
&c_i(d_{i+1}(f))+c_{i-1}(d_{i+1,i-1}(f))\\
&-\left(c_i(d_{i+1}(\mathcal{F}_{X_{i+1}}(f)))+c_{i-1}(d_{i+1,i-1}(\mathcal{F}_{X_{i+1}}(f)))\right)
\end{align*}
is invariant under $f \mapsto r_{i+1}(h)f$ for $h \in \SL_2(\A_F)$.
\end{lem}

\begin{proof}  
By Theorem \ref{thm:main} the quantity in the statement of the lemma is equal to 
\begin{align*}
&-\sum_{j=1}^{i+1}\sum_{\xi \in X_{j}^\circ(F)}I(d_{i+1,j}(f))(\xi)+\sum_{j=1}^{i+1}\sum_{\xi \in X_{j}^\circ(F)}I(d_{i+1,j}(\mathcal{F}_{X_{i+1}}(f)))(\xi)\\&-\kappa d_{i+1,0}(f)(0_{V_0},0,0)  +\kappa d_{i+1,0}(\mathcal{F}_{X_{i+1}}(f))(0_{V_0},0,0) \\& 
    -c_{i+1}(f)+c_{i+1}(\mathcal{F}_{X_{i+1}}(f)).
\end{align*}  
Using Lemma \ref{lem:EZ:case} the last two terms are invariant.  In view of the equivariance property of Lemma \ref{lem:F2:equiv} and the definition \eqref{di0} of $d_i$ the middle two terms are invariant.  Finally, the terms involving $I$ are invariant due to the equivariance properties just mentioned and the fact that $I$ is defined via an integral over $N(\A_F) \backslash \SL_2(\A_F).$
\end{proof}

\begin{proof}[Proof of Theorem \ref{thm:inv}]
If $ \dim V_i \not \in \{2,4\}$ then the theorem follows from Lemma \ref{lem:EZ:case}.  Thus we assume that $\dim V_i=4.$ We must show that $f \mapsto c_i(f)+c_{i-1}(d_i(f))$ is $r_i(\mathrm{SL}_2(\A_F))$-invariant.  
Let $f \in \mathcal{S}(V_i(\A_F) \oplus \A_F^2).$  Pick $\Phi \in \mathcal{S}(\A_F^2)$ such that $\Phi(0,0)=1$ and $\mathcal{F}_\wedge(\Phi)(0,0)=0.$  Then $\mathcal{F}_2^{-1}(f) \otimes \Phi \in \mathcal{S}(V_{i+1}(\A_F) \oplus \A_F^2)$, and by Lemma \ref{lem:F2:equiv} and Lemma \ref{lem:invariance} 
\begin{align*}
&c_i(d_{i+1}(\mathcal{F}_2^{-1}(f) \otimes \Phi))+c_{i-1}(d_{i+1,i-1}(\mathcal{F}_2^{-1}(f) \otimes \Phi))\\&-\left(c_i(d_{i+1}(\mathcal{F}_2^{-1}(f) \otimes \mathcal{F}_\wedge(\Phi)))+c_{i-1}(d_{i+1,i-1}(\mathcal{F}_2^{-1}(f) \otimes \mathcal{F}_\wedge (\Phi)))\right)
\end{align*}
is invariant under $f \otimes \Phi \mapsto r_{i}(h) \otimes L^\vee(h)(f \otimes \Phi)$ for $h\in \SL_2(\A_F)$.  But the above is 
\begin{align*}
&c_i(f)\Phi(0,0)+c_{i-1}(d_{i,i-1}(f))\Phi(0,0)\\&-c_i(f) \mathcal{F}_\wedge(\Phi)(0,0)-c_{i-1}(d_{i,i-1}(f))\mathcal{F}_\wedge(\Phi)(0,0)\\
&=c_i(f)+c_{i-1}(d_{i,i-1}(f))
\end{align*}
and we deduce the theorem.
\end{proof}

Motivated by Theorem \ref{thm:inv},
for $f \in \mathcal{S}(X_{\ell}(\A_F))$ we let
\begin{align} \label{ci:tilde}
\widetilde{c}_i(f):=\begin{cases} c_i(f) &\textrm{ if }\dim V_i >4\\
c_i(f)+c_{i-1}(d_i(f)) &\textrm{ if }\dim V_i=4\\
c_i(f)=0 &\textrm{ if } \dim V_i=2.
    \end{cases}
\end{align}
The maps $d_{i,i'}$ 
descend to $d_{i,i'}:\mathcal{S}(X_{i}(\A_F)) \to \mathcal{S}(X_{i'}(\A_F)).$
Moreover the Fourier transform 
$\mathcal{F}_{X_{\ell}}$ 
descends to 
\begin{align}
    \mathcal{F}_{X_{\ell}}: \mathcal{S}(X_{\ell}(\A_F)) \lto \mathcal{S}(X_{\ell}(\A_F)).
\end{align}
Combining theorems \ref{thm:main} and \ref{thm:inv} we obtain the following:
\begin{cor} \label{cor:main2} For $f \in \mathcal{S}(X_\ell(\A_F))$ one has that
\begin{align*}
&\sum_{i=1}^\ell\left(\widetilde{c}_i(d_{\ell,i}(f))+\sum_{\xi \in X^{\circ}_i(F)}I(d_{\ell,i}(f))(\xi)\right)+ \kappa d_{\ell,0}(f)(0_{V_0},0,0)\\&
=\sum_{i=1}^\ell \left(\widetilde{c}_i(d_{\ell,i}(\mathcal{F}_{X_{\ell}} (f)))+\sum_{\xi \in X^{\circ}_i(F)}I(d_{\ell,i}(\mathcal{F}_{X_{\ell}}(f)))(\xi)\right)+\kappa d_{\ell,0}(\mathcal{F}_{X_{\ell}}(f))(0_{V_0},0,0).
\end{align*}
\qed
\end{cor}

\section{Canonicity} \label{sec:canon}
As in \S \ref{ssec:canon}, a \textbf{framed flag of quadratic spaces extending  $V_0$} (or more briefly a framed flag) is a collection of vector spaces
\begin{align} \label{framed}
V_0<V_1'<\dots<V_{\ell}'
\end{align}
equipped with a nondegenerate quadratic form $Q'_\ell$ on $V_{\ell}'$ such that $Q_{\ell}'|_{V_0}=Q_0$
together with $v_{i}',w'_{i} \in V_{i}'(F) \cap V_{i-1}'^{\perp}(F) $ satisfying $Q'_{\ell}(v_i')=Q_{\ell}'(w_{i}')=0$ and $Q'_{\ell}(v_{i}'+w'_{i})=1.$  We call the ordered set $\{v_{i}',w_{i}'\}_{i=1}^{\ell}$ the \textbf{based polarization} of the flag.
By a \textbf{morphism} of framed flags
\begin{align} \label{phi}
\phi:V_0<V_1'<\dots<V_{\ell}' \lto V_0<V_1''<\dots<V_{\ell}''
\end{align}
 we mean a morphism of (additive) algebraic groups $\phi:V_\ell' \to V_\ell''$ that is the identity on $V_0$ and satisfies 
 $$
 (\phi(v_i'),\phi(w_i'))=
 (v_{i}'',w_i'')
 $$ 
 for $1 \leq i \leq \ell.$   Denote by $Q_{\ell}''$ the quadratic form on $V_{\ell}''.$ If $\phi:V_{\ell}' \to V_{\ell}''$ is a morphism of framed flags then for every $F$-algebra $R$ we have $Q''_{\ell}(\phi(v'))=Q'_{\ell}(v')$ for all $v' \in V'_{\ell}(R).$  
 
 By basic linear algebra, given any two framed flags $V_0<V_1'<\dots<V_\ell'$ and $V_0<V_1''<\dots<V_{\ell}''$ there is a unique morphism as in \eqref{phi}.  In other words:

\begin{lem}
Any object of the category of framed flags extending $V_0$ is initial. \qed
\end{lem}
To fix ideas, we define the universal framed flag $V_0 <V_1^{\mathrm{u}}<\dots<V_{\ell}^{\mathrm{u}}$ inductively by setting $(V_0^{\mathrm{u}},Q_0^{\mathrm{u}}):=(V_0,Q_0)$ and
$$
V_i^{\mathrm{u}}(R):=\{av_i^{\mathrm{u}}+bw_i^{\mathrm{u}}:a,b \in R\} \oplus V_{i-1}^{\mathrm{u}}(R).
$$
Here we equip $V_i^{\mathrm{u}}$ with the unique nondegenerate quadratic form $Q_i^{\mathrm{u}}$ such that $Q_i^{\mathrm{u}}|_{V_{i-1}^{\mathrm{u}}}=Q_{i-1}^{\mathrm{u}},$ $Q_i^{\mathrm{u}}(v_i^{\mathrm{u}})=Q_i^{\mathrm{u}}(w_i^{\mathrm{u}})=0$ and $Q_i^{\mathrm{u}}(v_i^{\mathrm{u}}+w_i^{\mathrm{u}})=1.$
Thus we have declared that $\{v_i^{\mathrm{u}},w_i^{\mathrm{u}}\}$ is a basis of the orthogonal complement of $V_{i-1}^{\mathrm{u}}$ in $V_i^{\mathrm{u}}$ with respect to $Q_i.$ 

Now assume that we are given a framed flag $V_0<V_1'<\dots<V_{\ell+1}'.$ Let $Q_{i}':=Q_\ell'|_{V_i}.$ We define
\begin{align*}
I:\mathcal{S}(V_{i}'(\A_F) \oplus \A_F^2) &\lto C^\infty(X'^\circ_i(\A_F))\\
c_i,\widetilde{c}_i:\mathcal{S}(V_{i}'(\A_F) \oplus \A_F^2) &\lto \CC
\end{align*}
as in \eqref{I}, \eqref{ci:def}, and \eqref{ci:tilde} using the operator  $r_i:=\rho_{Q_i',\psi} \oplus L^\vee.$  

Following \eqref{F2}, we define an $\SL_2(\A_F)$-equivariant map
\begin{align*}
\mathcal{F}_2:\mathcal{S}(V'_{i+1}(\A_F)) &\lto \mathcal{S}(V_i'(\A_F) \oplus \A_F^2)\\
f &\longmapsto (\xi,(a,b)) \mapsto \int_{\A_F} f(\xi+av_i'+xw_i')\psi(bx)dx
\end{align*}
and then
\begin{align*}
d_i:\mathcal{S}(V_{i}'(\A_F) \oplus \A_F^2) \lto \mathcal{S}(V_{i-1}'(\A_F) \oplus \A_F^2).
\end{align*}
is defined as in  \eqref{di0}.  We define $d_{i,i'}=d_{i,i-1} \circ \dots \circ d_{i'+1,i'}$ as before.  Let
\begin{align}
    \mathcal{S}(X'_{i}(\A_F)):=\mathcal{S}(V_{i}'(\A_F) \oplus \A_F^2)_{r_i(\SL_2(\A_F))}.
\end{align}
We point out that $I$ descends to $\mathcal{S}(X_{i}(\A_F)),$ $d_{i,i'}$ descends to a map $d_{i,i'}:\mathcal{S}(X'_{i}(\A_F)) \to \mathcal{S}(X'_{i'}(\A_F)),$ and $\widetilde{c}_i$ descends to a map $\widetilde{c}_i:\mathcal{S}(X'_i(\A_F)) \to \CC.$  

With these conventions, all of the results of this paper continue to hold if we replace $(V_i,X_i)$ by $(V_i',X_i')$ everywhere.

Now let 
\begin{align}
    \phi^{\mathrm{u}}:V_0<V_1^{\mathrm{u}}<\dots<V_{\ell}^{\mathrm{u}} \lto V_0<V_1'<\dots<V_\ell'
\end{align}
be the unique morphism of framed flags from the universal frame flag. 
We thus obtain a pullback morphisms $\phi^{\mathrm{u}*}:\mathcal{S}(V_{i}'(\A_F)) \to \mathcal{S}(V_{i}^{\mathrm{u}}(\A_F))$ and
\begin{align*}
    \phi^{\mathrm{u}*}:\mathcal{S}(V_{i}'(\A_F) \oplus \A_F^2) &\lto \mathcal{S}(V_{i}^{\mathrm{u}}(\A_F) \oplus \A_F^2)\\
    f &\longmapsto \left((\xi,(a,b)) \mapsto f(\phi^{\mathrm{u}}(\xi),(a,b))\right).
\end{align*}
We then define 
\begin{align*}
d_{\ell,i}^{\mathrm{u}}:=d_{\ell,i} \circ (\phi^{\mathrm{u}})^*: \mathcal{S}(V_{\ell}'(\A_F) \oplus \A_F^2)& \lto \mathcal{S}(V_{i}^{\mathrm{u}}(\A_F) \oplus \A_F^2),  & 0 \leq i \leq \ell-1\\c_i^{\mathrm{u}}:=c_i \circ d_{\ell,i}^{\mathrm{u}}, \widetilde{c}_i^{\mathrm{u}}:=\widetilde{c}_i \circ d_{\ell,i}^{\mathrm{u}}:\mathcal{S}(V_i^{\mathrm{u}}(\A_F) \oplus \A_F^2) &\lto \CC,  & 1 \leq i \leq \ell-1
\end{align*}
The map $d_{\ell,i}^{\mathrm{u}}$ descends to a map $d_{\ell,i}^{\mathrm{u}}:\mathcal{S}(X_{\ell}'(\A_F)) \to \mathcal{S}(X_{i}^{\mathrm{u}}(\A_F)),$ the linear functional $c_i^{\mathrm{u}}$ descends to a linear functional on $\mathcal{S}(X_i^{\mathrm{u}}(\A_F)),$ and the operators $I$ descend to operators $I:\mathcal{S}(X_i^{\mathrm{u}}(\A_F)) \to C^\infty(X_i^{\mathrm{u} \circ}(\A_F)).$

\begin{proof}[Proof of Corollary \ref{cor:main}]
As remarked before, Theorem \ref{thm:main} remains valid for $V_{\ell}^{\mathrm{u}}.$  Thus for any
$f \in \mathcal{S}(V_{\ell}'(\A_F) \oplus \A_F^2)$ we have 
\begin{align*}
&\sum_{i=1}^\ell\left(c_i(d_{\ell,i}(\phi^{\mathrm{u}*}(f)))+\sum_{\xi \in X^{\mathrm{u}\circ}_i(F)}I(d_{\ell,i}(\phi^{\mathrm{u}*}(f)))(\xi)\right)\\&+ \kappa d_{\ell,0}(\phi^{\mathrm{u}*}(f))(0_{V_0},0,0)\\&
=\sum_{i=1}^\ell \left(c_i(d_{\ell,i}(\mathcal{F}_{X_{\ell}^{\mathrm{u}}} (\phi^{\mathrm{u}*}(f))))+\sum_{\xi \in X^{\mathrm{u}\circ}_i(F)}I(d_{\ell,i}(\mathcal{F}_{X_{\ell}^{\mathrm{u}}}(\phi^{\mathrm{u}*}(f))))(\xi)\right)\\
&+\kappa d_{\ell,0}(\mathcal{F}_{X_{\ell}^{\mathrm{u}}}(\phi^{\mathrm{u}*}(f)))(0_{V_0},0,0).
\end{align*}
The definition of the Weil representation was recalled in \S \ref{ssec:WR}.  Using it, one checks that 
\begin{align}
    \rho_{Q^{\mathrm{u}}_i,\psi} \circ \phi^{\mathrm{u}*}=\phi^{\mathrm{u}*} \circ \rho_{Q_i',\psi}
\end{align}
It follows that 
$$
c_\ell(\phi^{\mathrm{u}*}(f))=c_{\ell}(f) \textrm{ and } I(\phi^{\mathrm{u}*}(f))(\xi)=I(f)(\phi^{\mathrm{u}}(\xi)).
$$
Moreover, using the definition of $\mathcal{F}_{X_i}$  from \eqref{FXi} we have $\mathcal{F}_{X_{\ell}^{\mathrm{u}}} \circ \phi^{\mathrm{u}*}=\phi^{\mathrm{u}*} \circ \mathcal{F}_{X_{\ell}}.$

The corollary follows.
\end{proof}

Combining corollary \ref{cor:main} and corollary \ref{cor:main2} we arrive at yet another corollary:

\begin{cor} \label{cor:inv}
    For  $f \in \mathcal{S}(X_\ell'(\A_F))$ one has that
\begin{align*}
& c_\ell(f)+\sum_{\xi \in X'^{\circ}_{\ell}(F)}I(f)(\xi)\\
&+\sum_{i=1}^{\ell-1}\left(\widetilde{c}_i^{\mathrm{u}}(d_{\ell,i}^{\mathrm{u}}(f))+\sum_{\xi \in X^{\mathrm{u}\circ}_i(F)}I(d_{\ell,i}^{\mathrm{u}}(f))(\xi)\right)+ \kappa d_{\ell,0}^{\mathrm{u}}(f)(0_{V_0},0,0)\\&
=c_\ell(\mathcal{F}_{X_\ell'}(f))+\sum_{\xi \in X'^{\circ}_{\ell}(F)}I(\mathcal{F}_{X_\ell'}(f))(\xi)\\
&+\sum_{i=1}^{\ell-1} \left(\widetilde{c}_i^{\mathrm{u}}(d_{\ell,i}^{\mathrm{u}}(\mathcal{F}_{X_{\ell}'} (f)))+\sum_{\xi \in X^{\mathrm{u}\circ}_i(F)}I(d_{\ell,i}^{\mathrm{u}}(\mathcal{F}_{X_{\ell}'}(f)))(\xi)\right)+\kappa d_{\ell,0}^{\mathrm{u}}(\mathcal{F}_{X_{\ell}'}(f))(0_{V_0},0,0).
\end{align*}
\qed
\end{cor}

\section{Application to summation over points on a quadric} \label{sec:circle:method}

Just as the Poisson summation formula gives a canonical method for estimating the number of points in a vector space of a given size, the summation formula of Theorem \ref{thm:main} gives a canonical method for estimating the number of points in a quadric of a given size.  Given the ubiquity of the Poisson summation formula in analytic number theory (and more generally in analysis) we can imagine many applications of 
Theorem \ref{thm:main}.  

Of course the question of analyzing the number of points on a quadric has a long history, with theta functions and the circle method playing a key role.  We point out in particular the author's work in \cite{GetzQuad} based on earlier work of Heath-Brown \cite{HBnewcircle} and Duke, Friedlander and Iwaniec \cite{DFI}.  To make the relationship between the circle method and Theorem \ref{thm:main} more transparent we explain Theorem \ref{thm:main} in a special case.  Let $F=\QQ$ and let $Q_\ell$ be a quadratic form with matrix $J_\ell$ defined as in \eqref{Qi}.  We assume that $J_\ell \in \GL_{V_\ell}(\ZZ)$ and $\det J_\ell=(-1)^{\ell}$.  Assume moreover that $\ell>0$ (which is to say that $V_\ell$ is isotropic) and that $\dim V_\ell > 4$.  The case $\dim V_\ell=4$ can also be treated, but the formula is slightly more complicated (see \cite{GetzQuad,HBnewcircle}). We choose 
$$
f=f_\infty \one_{V_{\ell+1}(\widehat{\ZZ})} \in \mathcal{S}(V_{\ell}(\A_\QQ) \oplus \A_\QQ^2)
$$  
and assume that 
$$
d_{\ell}(f_\infty)=c_\ell(f_\infty)=0.
$$
This is easy enough to arrange.
  Then for any $B \in \RR_{>0} \leq \QQ_{\infty}^{\times}$
\begin{align} \label{sum}
&\sum_{\xi \in X_{\ell}^\circ(\QQ)}I(f)\Big(\frac{\xi}{B}\Big)
=\sum_{n=0}^\infty\sum_{ \substack{\xi \in nV_\ell(\ZZ)-\{0\}\\
Q_\ell(\xi)=0
}} n^{\dim V_\ell/2-2}I(f_\infty)\Big(\frac{\xi}{B} \Big).
\end{align}
Here we have used Lemma \ref{lem:unr:conv}.
Thus one side of our summation formula is a weighted version of a sum familiar from analytic number theory.  It is a smoothed count of the number of integral zeros of $Q_\ell$ of size at most $B.$  

By Corollary \ref{cor:scaling} the other side is  
\begin{align*}
&B^{\dim V_\ell-2}c_\ell(\mathcal{F}_{X_\ell}(f))+B^{\dim V_\ell-2}
\sum_{n=0}^\infty\sum_{ \substack{\xi \in nV_\ell(\ZZ)-\{0\}\\
Q_\ell(\xi)=0
}} n^{\dim V_\ell/2-2}I(\mathcal{F}_{X_\ell}(f_\infty))\left(B\xi \right)\\
&
+B^{\dim V_\ell/2}\sum_{i=1}^{\ell-1}  \left(c_i(d_{\ell,i}(\mathcal{F}_{X_\ell}(f)))+
\sum_{n=0}^\infty\sum_{ \substack{\xi \in nV_i(\ZZ)-\{0\}\\
Q_i(\xi)=0
}} n^{\dim V_i/2-2}I(d_{\ell,i}(\mathcal{F}_{X_\ell}(f_\infty)))\left(\xi \right)\right)\\
&+B^{\dim V_\ell/2}\kappa d_{\ell,0}(\mathcal{F}_{X_\ell}(f))(0_{V_0},0,0).
\end{align*}

Thus Theorem \ref{thm:main} gives a complete asymptotic expansion of \eqref{sum} as a function of $B$ in terms of sums over the zero loci of all quadratic forms of rank less than or equal to $\dim V_\ell$ in the Witt class of $Q_\ell.$
 This goes far beyond what is obtained in the usual circle method, which usually only produces a description of the main term $c_{\ell}(\mathcal{F}_{X_\ell}(F))$ (but see \cite{GetzQuad,SchindlerHigherOrderExpansions,Tran,VaughanWooleyhoexp}).
The term $c_\ell(\mathcal{F}_{X_\ell}(f))$ is essentially the familiar singular series.
  The flexibility of choosing other test functions in $\mathcal{S}(X_\ell(\A_\QQ))$ allows one to impose congruence conditions on the sum.   It is important that one is allowed to choose arbitrary test functions at infinity for classical applications of this formula.  More specifically, if one imposed the conditions of \cite{BK:normalized}, for example, then the term $c_\ell(\mathcal{F}_{X_\ell}(f))$ would be zero, which would render the formula useless for the classical application of counting the number of points of size at most $B$ on $X_\ell(\ZZ)$.

% ----------------------------------------------------------------

\bibliography{refs}{}
\bibliographystyle{alpha}

\end{document}